\documentclass[11pt]{amsart}
\usepackage[latin1]{inputenc}
\usepackage{amsmath}
\usepackage{amsfonts}
\usepackage{amssymb}
\usepackage{graphicx}
\usepackage{fourier}
\usepackage{dsfont}
\usepackage{hyperref}
\usepackage{enumerate}
\usepackage{esint}

\newtheorem{theorem}{Theorem}[section]
\newtheorem{lemma}[theorem]{Lemma}

\theoremstyle{definition}
\newtheorem{definition}[theorem]{Definition}

\newtheorem{corollary}[theorem]{Corollary}
\theoremstyle{remark}
\newtheorem{remark}[theorem]{Remark}

\newcommand{\Norm}[1]{\left\lVert#1\right\rVert}
\newcommand{\abs}[1]{\left\lvert#1\right\rvert}
\newcommand{\pa}[1]{\left( #1 \right)}
\newcommand{\br}[1]{\left\lbrace #1\right\rbrace}
\newcommand{\spa}[2]{\mathbb{#1}^{#2}}
\newcommand{\kac}[0]{\mathbb{S}^{N-1}\left(\sqrt{N}\right)}
\newcommand{\R}{\mathbb{R}}
\newcommand{\N}{\mathbb{N}}
\numberwithin{equation}{section}

\author{Amit Einav}
\title{On The Subadditivity Of The Entropy On The Sphere}
\thanks{The author was supported by EPSRC grant EP/L002302/1}
\begin{document}

\maketitle
\begin{abstract}
We present a refinement of a known entropic inequality on the sphere, finding suitable conditions under which the uniform probability measure on the sphere behaves asymptomatically like the Gaussian measure on $\R^N$ with respect to the entropy. Additionally, we remark about the connection between this inequality and a the investigation of the many body Cercignani's conjecture.
\end{abstract}

\section{Introduction.}\label{sec: intro}
A fundamental principle in equilibrium statistical mechanics is that of the equivalence of ensembles. In mathematical terms, this principle states that the uniform measure on $\kac$, $d\sigma^N$, considered as a measure on $\R^N$ supported on the sphere, is close in behaviour to the Gaussian measure
\begin{equation}\nonumber
d\gamma_N(v)=\frac{e^{-\frac{\abs{v}^2}{2}}}{\pa{2\pi}^{\frac{N}{2}}}dv
\end{equation}
 when $N$ is very large. In this setting the uniform measure $d\sigma^N$ corresponds to the micro-canonical ensemble, representing a fixed number of particles with a fixed total energy, while the Gaussian measure $d\gamma_N$ corresponds to the canonical ensemble, representing a fixed number of particle in thermal equilibrium. For simple systems, the equivalence of ensembles principle means that for any finitely many number of particles with velocities $v_1,\dots,v_k$, $k\in\N$, and any observable function of those particles, $\phi\pa{v_1,\dots,v_k}$, the measurement of $\phi$ in the micro-canonical and canonical settings yields almost identical results, with a difference that converges to zero as the number of particles goes to infinity. In other words:
\begin{equation}\nonumber
\lim_{N\rightarrow\infty}\pa{\int_{\kac}\phi\pa{v_1,\dots,v_k}d\sigma^N - \int_{\R^N}\phi\pa{v_1,\dots,v_k}d\gamma_N}=0.
\end{equation}  
An acute difference between $d\sigma^N$ and $d\gamma_N$ may arise when one deals with quantities that depends on \emph{all} the particles in the ensemble, such as the case of the entropy, or more generally - the relative entropy, in non-equilibrium statistic mechanics. Such a deviation from the equivalence of ensembles principle was observed in \cite{CLL}, and will be described shortly. \\
We denote by $P\pa{X}$ the set of Borel probability measures on a Polish space X. Any measure in this current work will be assumed to be a Borel measure.
\begin{definition}\label{def: relative entropy}
Let $\mu,\nu \in P\pa{\R^d}$. The relative entropy of $\mu$ with respect to $\nu$ is defined as
\begin{equation}\nonumber
H\pa{\mu | \nu} =\begin{cases} \int_{\R^d} h\log h d\nu \quad h=\frac{d\mu}{d\nu} \\
\infty \quad\quad\quad\quad\quad \text{otherwise}.
\end{cases}
\end{equation}
\end{definition}
Note that we have not indicated the dimension of the underlying space in the notation of the relative entropy. It will be implicitly evident in all our discussions to follow.
\begin{definition}\label{def: entropy on the sphere}
Let $\mu\in P\pa{\kac}$ be absolutely continuous with respect to $d\sigma^N$ with a probability density function $F_N$. We denote by
\begin{equation}\nonumber
H_N\pa{F_N}=H\pa{F_N d\sigma^N | d\sigma^N}=H\pa{\mu | \sigma^N}.
\end{equation}
\end{definition}
Of special import to our work is the concept of marginals, and in particular first marginals.
\begin{definition}\label{def: k-th marginal}
Given $\mu\in P\pa{\R^d}$ we define its $k-$th marginal in the $\pa{i_1,\dots,i_k}-$th variables as the probability measure $\Pi_k^{\pa{i_1,\dots,i_k}}\pa{\mu}$ on $\R^k$ satisfying
\begin{equation}\label{eq: first marginal in jth variable}
\Pi_k^{\pa{i_1,\dots,i_k}}\pa{\mu}\pa{A_{1} \times \dots \times A_{k}}=\mu\pa{A^{\pa{i_1,\dots,i_k}}},
\end{equation}
where $A^{\pa{i_1,\dots,i_k}}=\widetilde{A_1}\times \dots \times \widetilde{A_N}$ with $\widetilde{A_j}=\begin{cases} A_l & j=i_l, l=1,\dots,k \\ \R & j\not=i_1,\dots,i_k \end{cases}$. 
\end{definition}
It is important to note that even if a probability measure, $\mu$, is supported on $\kac$, its $k-$th marginal is well defined on $\R^k$ whenever $k\leq N-1$ and is supported in the ball of radius $\sqrt{N}$ centred at the origin. Moreover, the $k-$th marginal in the $\pa{i_1,\dots,i_k}-$th variables of $\mu$ is absolutely continuous with respect to the Lebesgue measure on $\R^k$. We will denote by $\Pi_k^{\pa{i_1,\dots,i_k}}\pa{F_N}$ the probability density function of $\Pi_k^{\pa{i_1,\dots,i_k}}\pa{\mu}$.\\
 In what follows, whenever a measure $\mu$ will have a probability density (with respect to the Lebesgue or the uniform measure), $f$, we will use it interchangebly with $\mu$ in all our relevant quantities. For instance, writing $f\in P\pa{\R}$ or $H(f|g)$ will be an abusive notation to saying that the measure $\mu$ with density $f$ is in $P\pa{\R}$ or to computing the relative entropy of $d\mu(v)=f(v)dv$ with respect to the measure $g(v)dv$. \\
 We are now prepared to discuss the deviation from the equivalence of equilibrium principle, previously mentioned. It is simple to show (see the Appendix) that given $\mu\in P\pa{\R^N}$ such that $d\mu=F_N dv$, with $F_N$ having a finite second moment, one has that
\begin{equation}\label{eq: entropic inequality on R}
\sum_{j=1}^N H\pa{\Pi_1^{\pa{j}}\pa{F_N} | \gamma} \leq H\pa{F_N | \gamma_N},
\end{equation} 
where $\gamma=\gamma_1$. Trying to generalise (\ref{eq: entropic inequality on R}) one can define an appropriate first marginal on the sphere whenever $F_N$ is a probability density function on $\kac$ by
\begin{equation}\label{eq: first marginal in jth variable on the sphere}
F_j^{\pa{N}}(v)=\int_{\spa{S}{N-2}\pa{\sqrt{N-v^2}}}F_N d\sigma^{N-1}_{\sqrt{N-v^2}},
\end{equation}
where $d\sigma^k_r$ is the uniform probability measure on $\spa{S}{k-1}\pa{r}$. The expectation that (\ref{eq: entropic inequality on R}) will be approximately true on the sphere is false in general. It was proven in \cite{CLL} that
\begin{theorem}\label{thm: entropic inequality on the sphere bad}
Let $F_N\in P\pa{\kac}$. Then
\begin{equation}\label{eq: entropic inequality on the sphere}
\sum_{i=1}^N \int_{\spa{S}{N-1}\pa{\sqrt{N}}}F^{\pa{N}}_j\log F^{\pa{N}}_j d\sigma^N \leq 2 H_N\pa{F_N},
\end{equation}
and the constant $2$ is sharp.
\end{theorem}
The goal of the present work is to find sufficient conditions on the probability density $F_N$ on $\kac$ under which (\ref{eq: entropic inequality on R}) is indeed a good approximation to its the appropriate spherical analogue. The novelty of our approach is to incorporate elements from the theory of optimal transportation towards this goal. We define the quantities we shall use for the sake of completion.
\begin{definition}\label{def: wasserstein distance}
Let $X$ be a Polish space with a metric $d$ and let $\mu,\nu$ be two probability measures on $X$. For any $q\geq 1$ the Wasserstein distance of order $q$ between $\mu$ and $\nu$ is defined as
\begin{equation}\label{eq: wasserstein distance def}
W_q\pa{\mu,\nu}=\pa{\inf_{\pi\in\Pi\pa{\mu,\nu}}\int_{X\times X}d^q\pa{x,y}d\pi(x,y)}^{\frac{1}{q}},
\end{equation}
where $\Pi\pa{\mu,\nu}$, the space of coupling, is the space of all probability measures on $X\times X$ with marginals $\mu$ and $\nu$ respectively.
\end{definition}
\begin{definition}\label{def: relative fisher information}
Let $\mu,\nu\in P\pa{\R^d}$. The relative Fisher Information of $\mu$ with respect to $\nu$ is defined as
\begin{equation}\label{eq: relative fisher infroamtion}
I\pa{\mu|\nu}=\begin{cases} \int_{\R^d}\abs{\nabla \log h}^2 h d\nu \quad h=\frac{d\mu}{d\nu}\\
\infty \quad\quad\quad\quad\quad\quad\quad \text{otherwise.}
\end{cases}
\end{equation}
\end{definition}
One can extend the definition of the relative Fisher Information to $\kac$ in the case where $d\mu=F_N d\sigma^N$ and $d\nu=d\sigma^N$.
\begin{definition}\label{def: relative fisher information on the sphere}
Let $F_N\in P\pa{\kac}$. The Fisher Information of $F_N$ is defined as
\begin{equation}\label{eq: relative fisher infroamtion on the sphere}
I_N\pa{F_N}=I_N\pa{F_N d\sigma^N | d\sigma^N}=\int_{\kac}\abs{\nabla_{\mathbb{S}} \log F_N}^2 F_N d\sigma^N,
\end{equation}
where $\nabla_{\mathbb{S}}$ is the gradient on the sphere.
\end{definition}
For more information about optimal transportation, its tools and applications we refer the reader to the excellent \cite{Vtransport1} and \cite{Vtransport2}.\\
Last, but not least, for any measurable, non-negative function $f$  on $\R^d$ we denote by
\begin{equation}\label{eq: M_k def}
M_k\pa{f}=\int_{\R^d}\abs{v}^k f(v)dv
\end{equation}
the $k-$th moment of $f$.\\ 
The main theorems of this paper are:
\begin{theorem}\label{thm: improved entropic on sphere with conditions}
Let $F_N\in P\pa{\kac}$ such that there exists $k>2$ with
\begin{equation}\nonumber
\mathcal{A}_k=\sup_{N}\frac{\sum_{j=1}^N M_k \pa{\Pi_1^{\pa{j}}\pa{F_N}}}{N}<\infty.
\end{equation}
Assume in addition that
\begin{equation}\nonumber
\mathcal{A}_I=\sup_{N} \frac{\sum_{i=1}^N I\pa{\Pi_1^{\pa{j}}\pa{F_N}}}{N}<\infty,
\end{equation}
and that there exists $C_H>0$ such that
\begin{equation}\nonumber
\inf_{N}\frac{H_N\pa{F_N}}{N} \geq C_H.
\end{equation}
Then there exist $C_1,C_2>0$, independent of $N$ and $F_N$, such that for any $0<\beta<\frac{k}{2}-1$ and $1<p<\min\pa{\frac{k+1}{3},\frac{k}{2}}$
\begin{equation}\label{eq: improved entropic on sphere with conditions}
\begin{gathered}
\sum_{j=1}^N \int_{\kac}F_J^{\pa{N}}\log F_j^{\pa{N}}d\sigma^N \leq \Bigg(1 +\frac{C_1}{C_H N}\\
 \frac{4C_2\pa{1+\frac{C_k}{2}}^{\frac{1}{k}}}{C_H \pa{2N}^{\frac{1}{4}-\frac{1}{2k}}}\pa{\mathcal{A}_I -1 }^{\frac{1}{2}}
\pa{1+\mathcal{A}_k}^{\frac{1}{k}}+\frac{\mathcal{A}_k}{2C_HN^{\frac{k}{2}-1-\beta}}  \\
+\frac{C_p\mathcal{A}_I^{\frac{p-1}{2p}} \mathcal{A}_k^{\frac{1}{p}}}{2C_H\pa{1-\frac{1}{N^{\beta}}}^{\frac{k}{2p}}N^{\frac{1}{2}\pa{\frac{k+1}{p}-3}}} \Bigg) H_N\pa{F_N}
=\pa{1+\epsilon^{\pa{1}}_{H,I,k}(N)}H_N\pa{F_N},
\end{gathered}
\end{equation}
where $C_p=\pa{\int_{\abs{x}<1}\abs{\log\pa{1-x^2}}^{\frac{p}{p-1}}}^{\frac{p-1}{p}}$, and $C_k=\sup_{N}\pa{\frac{2}{N}}^{\frac{k}{2}}\frac{\Gamma\pa{\frac{N+k}{2}}}{\Gamma\pa{\frac{N}{2}}}$.
\end{theorem}
\begin{theorem}\label{thm: imrpoved entropic on sphere with spherical fisher information}
Let $F_N\in P\pa{\kac}$ such that there exists $k>2$ with
\begin{equation}\nonumber
\mathcal{A}_k=\sup_{N}\frac{\sum_{j=1}^N M_k \pa{\Pi_1^{\pa{j}}\pa{F_N}}}{N}<\infty.
\end{equation}
Assume in addition that there exists $2<q<k$ such that
\begin{equation}\nonumber
\mathcal{A}^P_{q}=\sup_{N}\frac{\sum_{j=1}^N P_{q}^{\pa{j}}\pa{F_N}}{N}<\infty.
\end{equation}
where 
\begin{equation}\nonumber
P_{q}^{\pa{j}}\pa{F_N}=\int_{\R} \frac{\Pi_1^{\pa{j}}\pa{F_N}(v)}{\pa{1-\frac{v^2}{N}}^{\frac{q}{q-2}}}dv,
\end{equation}
and that there exist constants $C_H,C_I>0$ such that
\begin{equation}\nonumber
\begin{gathered}
\inf_{N}\frac{H_N\pa{F_N}}{N} \geq C_H,\\
\sup_{N}\frac{I_N\pa{F_N}}{N} \leq C_I.
\end{gathered}
\end{equation}
Then there exist $C_1,C_2>0$, independent of $N$ and $F_N$, such that for any $0<\beta<\frac{k}{2}-1$ 
\begin{equation}\label{eq: imrpoved entropic on sphere with spherical fisher information}
\begin{gathered}
\sum_{j=1}^N \int_{\kac}F_J^{\pa{N}}\log F_j^{\pa{N}}d\sigma^N \leq \Bigg(1+\frac{C_1}{C_H N}\\
+\frac{C_2 2^{\frac{3}{2}+\frac{2}{q}}}{C_H}\pa{1+\frac{C_k}{2}}^{\frac{1}{k}}\pa{\pa{2C_I+2}^{\frac{q}{2(q-1)}}\pa{\mathcal{A}^P_{q}}^{\frac{q-2}{2(q-1)}}+2}^{\frac{q}{q-1}}
\frac{\pa{1+\mathcal{A}_k}^{\frac{1}{k}}}{\pa{2N}^{\frac{1}{2q}-\frac{1}{2k}}}
\\
 +\frac{\mathcal{A}_{k}}{2C_HN^{\frac{k}{2}-1-\beta}}  
+\frac{N}{2C_H(N-3)}\frac{\eta_{N,\beta}}{N^{\frac{k}{4}-\frac{1}{2}}\pa{1-\frac{1}{N^\beta}}^{\frac{k}{4}+\frac{1}{2}}}\pa{2C_I+2}^{\frac{1}{2}}\pa{\mathcal{A}_{k}}^{\frac{1}{2}} \Bigg)H_N\pa{F_N}\\
=\pa{1+\epsilon^{\pa{2}}_{H,I,k,q}(N)}H_N\pa{F_N},
\end{gathered}
\end{equation}
where, $C_k=\sup_{N}\pa{\frac{2}{N}}^{\frac{k}{2}}\frac{\Gamma\pa{\frac{N+k}{2}}}{\Gamma\pa{\frac{N}{2}}}$ and $\eta_{N\beta}=\sup_{x\in \left[0,N^{-\beta} \right]}x\pa{\log x}^2$.
\end{theorem}
We'd like to point out a difference between our theorems: Theorem \ref{thm: improved entropic on sphere with conditions} requires an average bound on the Fisher Information of the first marginals of $F_N$, a property that is not very intrinsic to the sphere. Theorem \ref{thm: imrpoved entropic on sphere with spherical fisher information}, on the other hand, relaxes this requirement and asks for information about the appropriate Fisher Information on the sphere. However, as the gradient on the sphere of any function of one variable $v_j$ is dampened near the poles $v_j=\pm \sqrt{N}$, additional control condition near the poles is needed, which is where $P_q^{\pa{j}}$ comes into play.\\
The idea of the proof of both theorems is to extend $F_N$ from the sphere to $\R^N$ where we are able to use (\ref{eq: entropic inequality on R}). We shall call this extension \emph{the Euclidean extension}. Once that is done one investigates the connection between the marginals of the extension of $F_N$ and $F_N$ using an appropriate distance (the Wasserstein distance) and associate the entropies of the appropriate marginals using an HWI theorem. The final step involves finding the connection between the entropy of the marginal and the entropy of the marginal on the sphere.\\
At this point we'd like to note the connection between inequality (\ref{eq: entropic inequality on the sphere}) and Kinetic Theory. Kac's model is a many particle random model which gives rise to a one dimensional Boltzmann-like equation (called the Kac-Boltzmann equation) as a mean field limit. Kac had hoped to use his model, whose complexity comes form the number of particles and not any non-linearity, to solve unknown questions for the associated Boltzmann equation, one of which was the rate of convergence to equilibrium. Kac suggested to use the $L^2$ distance and the associated spectral gap of the evolution operator to tackle this particular problem. While the spectral gap was proved to bounded from below uniformly in $N$ (Kac's conjecture), the $L^2$ distance was shown to be a catastrophic distance to consider under the setting of the model. A new distance, the relative entropy on the sphere, was investigated and with it the appropriate candidate for the rate of convergence: the entropy-entropy production ratio
\begin{equation}\nonumber
\Gamma_N=\inf_{F_N}\frac{D_N\pa{F_N}}{H_N\pa{F_N}},
\end{equation}
where $-D_N\pa{F_N}$ is obtained by differentiating the entropy under Kac's flow. For exponential decay of the entropy one would hope to show the existence of $C>0$, independent of $N$ such that $\Gamma_N \geq C$. This is called the many body Cercignani's conjecture. Unfortunately, in \cite{Villani} Villani has proven that
\begin{equation}\label{eq: entropy entropy production bad rate}
\Gamma_N \geq \frac{2}{N-1},
\end{equation}
using the heat semigroup on Kac's sphere, and conjectured that $\Gamma_N=O\pa{\frac{1}{N}}$, a claim that was essentially proved in \cite{Einav1}. Surprisingly, Carlen showed in \cite{Eric} that one can get (\ref{eq: entropy entropy production bad rate}) by using (\ref{eq: entropic inequality on the sphere}) and an inductive argument. The factor $2$ plays a crucial role in the proof, and one notices that if it was replaced with $1+\epsilon_N$, with $\epsilon_N$ converging to zero in a certain way, one wold get a lower bound that is independent of $N$! This was the main motivation behind the investigation of the presented work. For more information about Kac's model and the many body Cercignani's conjecture we refer the reader to  \cite{CCRLV,Einav1.5,HM,MM,Villani}.\\
The stricture of the paper is as follows: In Section \ref{sec: extension} we will describe the Euclidean extension, and see the connections between the first marginals and their moments, with respect to the original density. The entropic connection between the first marginals of the extension and the original density will be investigated in Section \ref{sec: entropy relation I}, while the entropic connection between the first marginals and the first marginals on the sphere will be shown in Section \ref{sec: the entropy relation II}. We will prove our main theorems in Section \ref{sec: proofs} and give a non trivial example for when the conditions of the theorems are satisfied in Section \ref{sec: example}. We then conclude the paper with a few final remarks in Section \ref{sec: final} and deal with a few technical computations in the Appendix.\\
\textbf{Acknowledgement:} We would like to greatly thank Eric Carlen for many discussions and insights on key ideas all along the progression of this work, without which this paper would never have seen the light of day. We would also like to offer our gratitude to Nathael Gozlan for providing us with a reference for the 'distorted' HWI inequality we use in Section \ref{sec: entropy relation I}, and Cl\'ement Mouhot for several discussions on the presented results.
\section{The Euclidean Extension and Marginal Relation.}\label{sec: extension}
The first step on the path to improve (\ref{eq: entropic inequality on the sphere}) is passing from the sphere to the Euclidean space. This is done by extending a given $F_N\in P \pa{\kac}$ to a function on $\R^N$, $\widetilde{F_N}$, in a way that is compatible with the entropy.
\begin{definition}\label{def: euclidean extension}
Given $F_N\in P\pa{\kac}$, its euclidean extension $\widetilde{F_N}$ is defined as 
\begin{equation}\label{eq: extension extension}
\widetilde{F_N}(v)=F_N\pa{\sqrt{N}\frac{v}{\abs{v}}}\cdot \gamma_N(v),
\end{equation}
with $v\in\mathbb{R}^N\setminus \br{0}$. 
\end{definition}
\begin{lemma}\label{lem: consistency}
$\widetilde{F_N}\in P\pa{\mathbb{R}^N}$ and 
\begin{equation}\nonumber
H\pa{\widetilde{F_N}|\gamma_N}=H_N\pa{F_N}.
\end{equation}
\end{lemma}
\begin{proof}
Using spherical coordinates, the fact that $F_N\pa{\sqrt{N}\frac{v}{\abs{v}}}$ depends only on the angular variable and the fact that $\gamma_N$ is radial we see that:
\begin{equation}\nonumber
\begin{gathered}
H\pa{\widetilde{F}_N | \gamma_N}=\int_{\mathbb{R}^N}F_N\pa{\sqrt{N}\frac{v}{\abs{v}}}\log\pa{F_N\pa{\sqrt{N}\frac{v}{\abs{v}}}}\gamma_N(v)dv \\
=\pa{\int_{\mathbb{S}^{N-1}}F_N\pa{\sqrt{N}\frac{v}{\abs{v}}}\log\pa{F_N\pa{\sqrt{N}\frac{v}{\abs{v}}}}d\sigma^N_1}\pa{\abs{\mathbb{S}^{N-1}}\int_{0}^\infty \frac{r^{N-1}}{\pa{2\pi}^{\frac{N}{2}}}e^{-\frac{r^2}{2}}dr}\\
=H_N\pa{F_N}, 
\end{gathered}
\end{equation}
since 
\begin{equation}\label{eq: gamma integration}
1=\int_{\mathbb{R}^N}\gamma_N(v)dv =\abs{\mathbb{S}^{N-1}}\int_{0}^\infty \frac{r^{N-1}}{\pa{2\pi}^{\frac{N}{2}}}e^{-\frac{r^2}{2}}dr. 
\end{equation} 
Using the same argument one can easily show that $\widetilde{F_N}$ is indeed a probability density.
\end{proof}
Now that we have a possible extension at hand, the next step we'd like to explore is the relation between its first marginals and those of the original function. We start by recalling the following simple Fubini-Tonelli type theorem on the sphere (see \cite{Einav1} for instance):
\begin{equation}\label{eq: sphere integration}
\begin{gathered}
\int_{\mathbb{S}^{N-1}\pa{r}} F_N d\sigma_r^N = \frac{\abs{\mathbb{S}^{N-k-1}}}{\abs{\mathbb{S}^{N-1}}}\frac{1}{r^{N-2}} \int_{\sum_{i=1}^k v_i^2 \leq r^2}\pa{r^2-\sum_{i=1}^k v_i^2}^{\frac{N-k-2}{2}} \\
\pa{\int_{\mathbb{S}^{N-k-1}\pa{\sqrt{r^2-\sum_{i=1}^k v_i^2}}}F_N d\sigma^{N-k}_{\sqrt{r^2-\sum_{i=1}^k v_i^2}}}dv_1 \dots dv_k.
\end{gathered}
\end{equation}
Formula (\ref{eq: sphere integration}) allows us to write a concrete expression to the $k-$th marginal of a probability density function $F_N \in P\pa{\kac}$ in its $\pa{i_1,\dots,i_k}$ variables whenever $k\leq N-1$. Indeed, one easily see that
\begin{equation}\label{eq: normal marginal formula}
\begin{gathered}
\Pi^{\pa{i_1,\dots,i_k}}_k\pa{F_N}(v_{i_1},\dots,v_{i_k})=\frac{\abs{\mathbb{S}^{N-k-1}}}{\abs{\mathbb{S}^{N-1}}}\frac{1}{N^{\frac{k}{2}}}\\
\pa{1-\frac{\sum_{l=1}^k v_{i_l}^2}{N}}_{+}^{\frac{N-k-2}{2}} 
\pa{\int_{\mathbb{S}^{N-k-1}\pa{\sqrt{N-\sum_{i=1}^k v_{i_l}^2}}}F_N d\sigma^{N-k}_{\sqrt{N-\sum_{i=1}^k v_{i_l}^2}}}.
\end{gathered}
\end{equation}
Using this, we can conclude the following:
\begin{lemma}\label{lem: general marginal formula}
Let $F_N\in P\pa{\spa{S}{N-1}\pa{\sqrt{N}}}$. Then, the $k-$th marginal of $\widetilde{F_N}$ in the $\pa{i_1,\dots,i_k}$ variables is given by
\begin{equation}\label{eq: general marginal formula}
\begin{gathered}
\Pi_k^{\pa{i_1,\dots,i_k}}\pa{\widetilde{F}_N}(v_1,\dots,v_k) \\
= \abs{\mathbb{S}^{N-1}}N^{\frac{k}{2}}\int_0 ^\infty \frac{r\pa{\sum_{l=1}^k v_{i_l}^2 +r^2}^{\frac{N-k-2}{2}}}{\pa{2\pi}^{\frac{N}{2}}}e^{-\frac{r^2+\sum_{l=1}^k v_{i_l}^2}{2}} \\
\Pi_k^{\pa{i_1,\dots,i_k}}\pa{F_N}\pa{\frac{\sqrt{N}v_{i_1}}{\sqrt{\sum_{l=1}^k v_{i_l}^2 +r^2}},\dots,\frac{\sqrt{N}v_{i_k}}{\sqrt{\sum_{l=1}^k v_{i_l}^2 +r^2}}}dr.
\end{gathered}
\end{equation}
\end{lemma}
\begin{proof}
By its definition
\begin{equation}\nonumber
\Pi_k^{\pa{i_1,\dots,i_k}}\pa{\widetilde{F}_N}(v_{i_1},\dots,v_{i_k}) = \int_{\mathbb{R}^{N-k}}F_N\pa{\sqrt{N}\frac{v}{\abs{v}}}\gamma_{N}\pa{v_1,\dots,v_N}d\tilde{v}_{i_1,\dots,v_{i_k}}
\end{equation}
where $d\tilde{v}_{i_1,\dots,v_{i_k}}$ represents $dv$ excluding $dv_{i_1}\dots dv_{i_k}$. For the sake of simplicity of notations we'll assume that $i_l=l$. We find that
\begin{equation}\nonumber
\begin{gathered}
\Pi_k^{\pa{1,\dots,k}}\pa{\widetilde{F}_N}(v_{1},\dots,v_{k}) =\\
\int_{\mathbb{R}^{N-k}}F_N\pa{\frac{\sqrt{N} v_1}{\sqrt{\sum_{i=1}^k v_i^2 +\sum_{i=k+1}^N v_i^2 }},\dots,\frac{\sqrt{N} v_N}{\sqrt{\sum_{i=1}^k v_i^2 +\sum_{i=k+1}^N v_i^2 }}}
\frac{e^{-\frac{\sum_{i=1}^k v_i^2 + \sum_{i=k+1}^N v_i^2}{2}}}{\pa{2\pi}^{\frac{N}{2}}}dv_{k+1}\dots dv_N\\
\end{gathered}
\end{equation}
\begin{equation}\nonumber
\begin{gathered}
= \int_{\mathbb{S}^{N-k-1}\times [0,\infty)}F_N\pa{\frac{\sqrt{N} v_1}{\sqrt{\sum_{i=1}^k v_i^2 +r^2}},\dots,\frac{\sqrt{N} v_k}{\sqrt{\sum_{i=1}^k v_i^2 +r^2}},\frac{\sqrt{N}r\Omega}{\sqrt{\sum_{i=1}^k v_i^2 +r^2}}}
\frac{r^{N-k-1}e^{-\frac{\sum_{i=1}^k v_i^2 + r^2}{2}}}{\pa{2\pi}^{\frac{N}{2}}}drd\Omega.\\
=\abs{\mathbb{S}^{N-k-1}} \int_{[0,\infty)}dr \frac{r^{N-k-1}e^{-\frac{\sum_{i=1}^k v_i^2 + r^2}{2}}}{\pa{2\pi}^{\frac{N}{2}}} \\
\end{gathered}
\end{equation}
\begin{equation}\nonumber
\pa{\int_{\mathbb{S}^{N-k-1}\pa{\frac{\sqrt{N}r}{\sqrt{\sum_{i=1}^k v_i^2 + r^2}}}}F_N\pa{\frac{\sqrt{N} v_1}{\sqrt{\sum_{i=1}^k v_i^2 +r^2}},\dots,\frac{\sqrt{N} v_k}{\sqrt{\sum_{i=1}^k v_i^2 +r^2}},\sigma}
 d\sigma^{N-k}_{\mathbb{S}^{N-k-1}\pa{\frac{\sqrt{N}r}{\sqrt{\sum_{i=1}^k v_i^2 +r^2}}}}}
\end{equation}
Since
\begin{equation}\nonumber
N-N\sum_{i=1}^k \frac{v_i^2}{\sum_{i=1}^k v_i^2+r^2}=\frac{Nr^2}{\sum_{i=1}^k v_i^2+r^2},
\end{equation}
we have that
\begin{equation}\nonumber
\begin{gathered}
\int_{\mathbb{S}^{N-k-1}\pa{\frac{\sqrt{N}r}{\sqrt{\sum_{i=1}^k v_i^2 + r^2}}}}F_N\pa{\frac{\sqrt{N} v_1}{\sqrt{\sum_{i=1}^k v_i^2 +r^2}},\dots,\frac{\sqrt{N} v_k}{\sqrt{\sum_{i=1}^k v_i^2 +r^2}},\sigma}
 d\sigma^{N-k}_{\mathbb{S}^{N-k-1}\pa{\frac{\sqrt{N}r}{\sqrt{\sum_{i=1}^k v_i^2 +r^2}}}}\\
 =\frac{\abs{\mathbb{S}^{N-1}}N^{\frac{k}{2}}}{\abs{\mathbb{S}^{N-k-1}}}
\frac{\Pi_k^{\pa{1,\dots,k}}\pa{F_N}\pa{\frac{\sqrt{N}v_1}{\sqrt{\sum_{i=1}^k v_i^2 +r^2}},\dots,\frac{\sqrt{N}v_k}{\sqrt{\sum_{i=1}^k v_i^2 +r^2}}}}
{\pa{1-\frac{N\pa{\frac{\sum_{i=1}^k v_i^2}{\sum_{i=1}^k v_i^2 +r^2}}}{N}}_{+}^{\frac{N-k-2}{2}}}\\
=\frac{\abs{\mathbb{S}^{N-1}}N^{\frac{k}{2}}}{\abs{\mathbb{S}^{N-k-1}}}
\frac{\pa{\sum_{i=1}^k v_i^2 +r^2}^{\frac{N-k-2}{2}}\Pi_k^{\pa{1,\dots,k}}\pa{F_N}\pa{\frac{\sqrt{N}v_1}{\sqrt{\sum_{i=1}^k v_i^2 +r^2}},\dots,\frac{\sqrt{N}v_k}{\sqrt{\sum_{i=1}^k v_i^2 +r^2}}}}{r^{N-k-2}}.
\end{gathered}
\end{equation}
Combining the above equalities yields the desired result.
\end{proof}
Of particular interest is the case of the first marginal in the $j-$th variable, $\Pi_1^{\pa{j}}\pa{\widetilde{F_N}}$. Using Lemma \ref{lem: general marginal formula} we obtain
\begin{corollary}\label{cor: first marginal connection}
Let $F_N\in P\pa{\spa{S}{N-1}\pa{\sqrt{N}}}$. Then
\begin{equation}\label{eq: first marginal connection}
\Pi_1^{\pa{j}}\pa{\widetilde{F_N}}\pa{v}=\frac{\abs{\mathbb{S}^{N-1}}N^{\frac{N}{2}}}{\pa{2\pi}^{\frac{N}{2}}}\int_0^{\text{sgn}(v)\sqrt{N}}\Pi_1^{\pa{j}}\pa{F_N}(x) \frac{v^{N-1}}{x^N}e^{-\frac{Nv^2}{2x^2}}dx \\.
\end{equation}
\end{corollary}
\begin{proof}
From (\ref{eq: general marginal formula}) we know that
\begin{equation}\nonumber
\begin{gathered}
\Pi_1^{\pa{j}}\pa{\widetilde{F}_N}(v) = \abs{\mathbb{S}^{N-1}}\sqrt{N} 
\int_0 ^\infty \Pi_1^{\pa{j}}\pa{F_N}\pa{\frac{\sqrt{N}v}{\sqrt{v^2+r^2}}}\frac{r\pa{v^2 +r^2}^{\frac{N-3}{2}}}{\pa{2\pi}^{\frac{N}{2}}}e^{-\frac{r^2+v^2}{2}}dr.
\end{gathered}
\end{equation}
Using the change of variables $x=\frac{\sqrt{N}v}{\sqrt{v^2+r^2}}$ we find that
\begin{equation}\nonumber
r^2=\frac{v^2\pa{N-x^2}}{x^2}=\frac{Nv^2}{x^2}-v^2,
\end{equation}
or
\begin{equation}\nonumber
r=\frac{\abs{v}\sqrt{N-x^2}}{\abs{x}}=\frac{v\sqrt{N-x^2}}{x},
\end{equation}
as the sign of $v$ and $x$ are the same. Since
\begin{equation}\nonumber
\begin{gathered}
r=0 \Longrightarrow x=\text{sgn}(v)\sqrt{N},
\end{gathered}
\end{equation}
\begin{equation}\nonumber
\begin{gathered}
r=\infty \Longrightarrow x=0,
\end{gathered}
\end{equation}
\begin{equation}\nonumber
\begin{gathered}
rdr=-\frac{Nv^2}{x^3}dx.
\end{gathered}
\end{equation}
we conclude that
\begin{equation}\nonumber
\begin{gathered}
\Pi_1^{\pa{j}}\pa{\widetilde{F}_N}(v) =\abs{\mathbb{S}^{N-1}}\sqrt{N}
\int_0^{\text{sgn}(v)\sqrt{N}} \Pi_1^{\pa{j}}\pa{F_N}(x)\pa{\frac{Nv^2}{x^2}}^{\frac{N-3}{2}}\frac{Nv^2}{x^3\pa{2\pi}^{\frac{N}{2}}}e^{-\frac{Nv^2}{2x^2}}dx\\
\end{gathered}
\end{equation}
\begin{equation}\nonumber
\begin{gathered}
=\frac{\abs{\mathbb{S}^{N-1}}N^{\frac{N}{2}}}{\pa{2\pi}^{\frac{N}{2}}}\int_0^{\text{sgn}(v)\sqrt{N}}\Pi_1^{\pa{j}}\pa{F_N}(x) \frac{\abs{v}^{N-1}}{x\abs{x}^{N-1}}e^{-\frac{Nv^2}{2x^2}}dx \\
=\frac{\abs{\mathbb{S}^{N-1}}N^{\frac{N}{2}}}{\pa{2\pi}^{\frac{N}{2}}}\int_0^{\text{sgn}(v)\sqrt{N}}\Pi_1^{\pa{j}}\pa{F_N}(x) \frac{v^{N-1}}{x^N}e^{-\frac{Nv^2}{2x^2}}dx, \\
\end{gathered}
\end{equation}
completing the proof.
\end{proof}
An interesting application of Corollary \ref{cor: first marginal connection} is a moment connection between $\Pi_1^{\pa{j}}\pa{F_N}$ and $\Pi_1^{\pa{j}}\pa{\widetilde{F_N}}$.
\begin{lemma}\label{lem: moment connection}
Let $F_N\in P\pa{\spa{S}{N-1}\pa{\sqrt{N}}}$. Then
\begin{equation}\label{eq: moment connection}
\int_{\mathbb{R}}|v|^m\Pi_1^{\pa{j}}\pa{\widetilde{F}_N}(v)dv= \pa{\frac{2}{N}}^{\frac{m}{2}}\frac{\Gamma\pa{\frac{N+m}{2}}}{\Gamma\pa{\frac{N}{2}}}\int_{-\sqrt{N}}^{\sqrt{N}}|v|^m\Pi_1^{\pa{j}}\pa{F_N}(v)dv.
\end{equation}
\end{lemma}
\begin{proof}
Using (\ref{eq: first marginal connection}), we have that
\begin{equation}\nonumber
\begin{gathered}
\int_{\mathbb{R}}|v|^m\Pi_1\pa{\widetilde{F}_N}(v)dv \\
=\frac{\abs{\mathbb{S}^{N-1}}N^{\frac{N}{2}}}{\pa{2\pi}^{\frac{N}{2}}}\int_{\mathbb{R}}\int_0^{\text{sgn}(v)\sqrt{N}}\Pi_1\pa{F_N}(x) \frac{v^{N-1}|v|^m}{x^N}e^{-\frac{Nv^2}{2x^2}}dxdv\\
\end{gathered}
\end{equation}
\begin{equation}\nonumber
\begin{gathered}
=\frac{\abs{\mathbb{S}^{N-1}}N^{\frac{N}{2}}}{\pa{2\pi}^{\frac{N}{2}}}\Bigg(\int_0^\infty\int_0^{\sqrt{N}}\Pi_1\pa{F_N}(x) \frac{v^{N+m-1}}{x^N}e^{-\frac{Nv^2}{2x^2}}dxdv\\
-\int_{-\infty}^0 \int_{-\sqrt{N}}^0 \Pi_1\pa{F_N}(x) \frac{(-1)^mv^{N+m-1}}{x^N}e^{-\frac{Nv^2}{2x^2}}dxdv\Bigg)\\
\end{gathered}
\end{equation}
\begin{equation}\nonumber
\begin{gathered}
\underset{y=\frac{\sqrt{N}}{x}v}{=}\frac{\abs{\mathbb{S}^{N-1}}N^{\frac{N}{2}}}{\pa{2\pi}^{\frac{N}{2}}}\Bigg(\int_0^\infty\int_0^{\sqrt{N}}\Pi_1\pa{F_N}(x) \frac{x^m y^{N+m-1}}{N^{\frac{N+m}{2}}}e^{-\frac{y^2}{2}}dxdy\\
-\int_{\infty}^0 \int_{-\sqrt{N}}^0 \Pi_1\pa{F_N}(x) \frac{(-1)^m x^m y^{N+m-1}}{N^{\frac{N+m}{2}}}e^{-\frac{y^2}{2}}dxdy\Bigg)\\
\end{gathered}
\end{equation}
\begin{equation}\nonumber
\begin{gathered}
=\frac{\abs{\mathbb{S}^{N-1}}\pa{2\pi}^{\frac{m}{2}}}{N^{\frac{m}{2}}}\int_{-\sqrt{N}}^{\sqrt{N}}|x|^m \Pi_1\pa{F_N}(x)\pa{\frac{1}{\pa{2\pi}^{\frac{N+m}{2}}}\int_{0}^\infty y^{N+m-1}e^{-\frac{y^2}{2}}dy}dx\\ 
= \frac{\abs{\mathbb{S}^{N-1}}\pa{2\pi}^{\frac{m}{2}}}{N^{\frac{m}{2}}\abs{\mathbb{S}^{N+m-1}}}\int_{-\sqrt{N}}^{\sqrt{N}}|x|^m\Pi_1\pa{F_N}(x)dx.
\end{gathered}
\end{equation}
The result follows from the formula 
\begin{equation}\nonumber
\abs{\spa{S}{N-1}}=\frac{2\pi^{\frac{N}{2}}}{\Gamma\pa{\frac{N}{2}}}.
\end{equation}
\end{proof}
Lemma \ref{lem: moment connection} implies the following:
\begin{corollary}\label{cor: control of moemnt of extension via moments of original}
For any $k>0$ there exists $C_k>0$, independent of $N$, such that
\begin{equation}\label{eq: control of moemnt of extension via moments of original}
M_k\pa{\Pi_1^{\pa{j}}\pa{\widetilde{F_N}}} \leq C_k M_k\pa{\Pi_1^{\pa{j}}\pa{F_N}}.
\end{equation}
Moreover, when $k=2$ there is equality in (\ref{eq: control of moemnt of extension via moments of original}) with $C_2=1$.
\end{corollary}
\begin{proof}
From (\ref{eq: moment connection}) we see that choosing
\begin{equation}\nonumber
C_k=\sup_N \pa{\frac{2}{N}}^{\frac{k}{2}}\frac{\Gamma\pa{\frac{N+k}{2}}}{\Gamma\pa{\frac{N}{2}}}
\end{equation}
proves the claim. Since $\Gamma(z)=z^{z-\frac{1}{2}}e^{-z}\sqrt{2\pi}\pa{1+\frac{1}{12z}+\dots}$ as $z$ approaches infinity, we have that
\begin{equation}\nonumber
\begin{gathered}
\frac{\Gamma\pa{\frac{N+k}{2}}}{\Gamma\pa{\frac{N}{2}}}=\frac{\pa{\frac{N}{2}}^{\frac{N+k-1}{2}}\pa{1+\frac{k}{N}}^{\frac{N+k-1}{2}}e^{-\frac{N+k}{2}}\sqrt{2\pi}\pa{1+\frac{1}{6(N+k)}+\dots}}{\pa{\frac{N}{2}}^{\frac{N-1}{2}}e^{-\frac{N}{2}}\sqrt{2\pi}\pa{1+\frac{1}{6N}+\dots}}\\
=\pa{\frac{N}{2}}^{\frac{k}{2}}\pa{1+\frac{k}{N}}^{\frac{N+k-1}{2}}e^{-\frac{k}{2}}\frac{1+\frac{1}{6(N+k)}+\dots}{1+\frac{1}{6N}+\dots},
\end{gathered}
\end{equation}
showing that $C_k$ is indeed finite. Lastly, If $k=2l$ then
\begin{equation}\nonumber
\Gamma\pa{\frac{N+k}{2}}=\Gamma\pa{\frac{N}{2}+l}=\pa{\Pi_{i=0}^{l-1}\pa{\frac{N}{2}-i}}\Gamma\pa{\frac{N}{2}}. 
\end{equation}
In this case,
\begin{equation}\nonumber
M_k\pa{\Pi_1^{\pa{j}}\pa{\widetilde{F_N}}} =\pa{\Pi_{i=0}^{\frac{k}{2}-1}\pa{1-\frac{2i}{N}}} M_k\pa{\Pi_1^{\pa{j}}\pa{F_N}}\leq M_k\pa{\Pi_1^{\pa{j}}\pa{F_N}},
\end{equation}
with equality if and only if $k=2$.
\end{proof}
\section{The Entropy Relation - From Marginals To The Marginals Of The Extension.}\label{sec: entropy relation I}
Now that we have managed to extend our probability density from $\kac$ to $\R^N$ we would like to find out how much information we may have 'lost' during that process, at least in the sense of the entropy functional. The main theoretical tool to connect between the two will be the HWI inequality (see \cite{C-E,Vtransport1,Vtransport2}). In this section we will slowly investigate the quantities that will play a role in the final connection between the entropies, namely the Wasserstein distance and the Fisher Information, and eventually quantify the 'loss' in the transition followed by our extension.
\begin{lemma}\label{lem: wasserstein 1 distance}
Let $F_N\in P\pa{\spa{S}{N-1}\pa{\sqrt{N}}}$. Then for any $j=1,\dots,N$
\begin{equation}\label{eq: wasserstein 1 distance}
W_1\pa{\Pi_1^{\pa{j}}\pa{F_N},\Pi_1^{\pa{j}}\pa{\widetilde{F_N}}} \leq \frac{M_2\pa{\Pi_1^{\pa{j}}\pa{F_N}}^{\frac{1}{2}}}{\sqrt{2N}}\pa{1+\tau_N},
\end{equation}
where $\tau_N\underset{N\rightarrow\infty}{\longrightarrow}0$ as $N$ goes to infinity, is given explicitly and independently of $F_N$.
\end{lemma}
\begin{proof}
The proof relies on the famous Kantorovich-Rubinstein formula (see \cite{Vtransport1}): For any $\mu,\nu\in P\pa{X}$, where $X$ is a Polish space,
\begin{equation}\nonumber
W_1\pa{\mu,\nu}=\sup \pa{\int_{X}\psi(x)\pa{d\mu-d\nu}(x)},
\end{equation} 
where the supremum is taken over all $1-$Lipschitz functions $\psi$.\\ 
For any $\phi\in C_b\pa{\mathbb{R}}$ we find that
\begin{equation}\nonumber
\begin{gathered}
\int_\mathbb{R} \phi(v)\Pi_1^{\pa{j}}\pa{\widetilde{F_N}}dv = \frac{\abs{\spa{S}{N-1}}}{\pa{2\pi}^{\frac{N}{2}}}N^{\frac{N}{2}}\int_{0}^{\infty}\int_{0}^{\sqrt{N}}\phi(v) \Pi_1^{\pa{j}}\pa{F_N}(x)\frac{v^{N-1}}{x^N}e^{-\frac{Nv^2}{2x^2}}dvdx\\ 
 +\frac{\abs{\spa{S}{N-1}}}{\pa{2\pi}^{\frac{N}{2}}}N^{\frac{N}{2}}\int_{-\infty}^0\int_{-\sqrt{N}}^0 \phi(v) \Pi_1^{\pa{j}}\pa{F_N}(x)\pa{-\frac{v^{N-1}}{x^N}}e^{-\frac{Nv^2}{2x^2}}dvdx \\
 \end{gathered}
\end{equation}
 \begin{equation}\nonumber
\begin{gathered}
 \underset{y=\frac{\sqrt{N}v}{x}}{=}\frac{\abs{\spa{S}{N-1}}}{\pa{2\pi}^{\frac{N}{2}}}\int_{0}^{\infty}\int_{0}^{\sqrt{N}}\phi\pa{\frac{yx}{\sqrt{N}}} \Pi_1^{\pa{j}}\pa{F_N}(x)y^{N-1}e^{-\frac{y^2}{2}}dydx\\
+ \frac{\abs{\spa{S}{N-1}}}{\pa{2\pi}^{\frac{N}{2}}}\int_{\infty}^0 \int_{-\sqrt{N}}^0\phi\pa{\frac{yx}{\sqrt{N}}} \Pi_1^{\pa{j}}\pa{F_N}(x)\pa{-y^{N-1}}e^{-\frac{y^2}{2}}dydx \\
\end{gathered}
\end{equation}
\begin{equation}\nonumber
\begin{gathered}
 =\frac{\abs{\spa{S}{N-1}}}{\pa{2\pi}^{\frac{N}{2}}}\int_{0}^{\infty}\int_{-\sqrt{N}}^{\sqrt{N}}\phi\pa{\frac{yx}{\sqrt{N}}} \Pi_1^{\pa{j}}\pa{F_N}(x)y^{N-1}e^{-\frac{y^2}{2}}dydx.
\end{gathered}
\end{equation}
Since $\int_{0}^\infty y^{N-1}e^{-\frac{y^2}{2}}dy = \frac{\pa{2\pi}^{\frac{N}{2}}}{\abs{\spa{S}{N-1}}}$ and $\Pi_1^{\pa{j}}\pa{F_N}$ is supported in $[-\sqrt{N},\sqrt{N}]$ we see that
\begin{equation}\nonumber
\begin{gathered}
\abs{\int_\mathbb{R} \phi(x)\Pi_1^{\pa{j}}\pa{F_N}(x)dx - \int_{\mathbb{R}}\phi(v)\Pi_1^{\pa{j}}\pa{\widetilde{F_N}}(v)dv} \\
\leq  \frac{\abs{\spa{S}{N-1}}}{\pa{2\pi}^{\frac{N}{2}}}\int_{0}^{\infty}\int_{-\sqrt{N}}^{\sqrt{N}}\abs{\phi\pa{\frac{yx}{\sqrt{N}}}-\phi(x)} \Pi_1^{\pa{j}}\pa{F_N}(x)y^{N-1}e^{-\frac{y^2}{2}}dydx.
\end{gathered}
\end{equation}
If in addition $\phi$ is $1-$Lipshitz we have that
\begin{equation}\nonumber
\begin{gathered}
\abs{\int_\mathbb{R} \phi(x)\Pi_1^{\pa{j}}\pa{F_N}(x)dx - \int_{\mathbb{R}}\phi(v)\Pi_1^{\pa{j}}\pa{\widetilde{F_N}}(v)dv} \\
\leq  \frac{\abs{\spa{S}{N-1}}}{\pa{2\pi}^{\frac{N}{2}}}\int_{0}^{\infty}\int_{-\sqrt{N}}^{\sqrt{N}}\abs{\frac{yx}{\sqrt{N}}-x} \Pi_1^{\pa{j}}\pa{F_N}(x)y^{N-1}e^{-\frac{y^2}{2}}dydx \\
\end{gathered}
\end{equation}
\begin{equation}\nonumber
\begin{gathered}
=\pa{\int_{-\sqrt{N}}^{\sqrt{N}}\abs{x}\Pi_1^{\pa{j}}\pa{F_N}(x)dx}\pa{ \frac{\abs{\spa{S}{N-1}}}{\pa{2\pi}^{\frac{N}{2}}}\int_{0}^{\infty}\abs{\frac{y}{\sqrt{N}}-1} y^{N-1}e^{-\frac{y^2}{2}}dydx}\\
\leq M_2\pa{\Pi_1^{\pa{j}}\pa{F_N}}^{\frac{1}{2}}\pa{\frac{\abs{\spa{S}{N-1}}}{\pa{2\pi}^{\frac{N}{2}}}\int_{0}^{\infty}\pa{\frac{y}{\sqrt{N}}-1}^2 y^{N-1}e^{-\frac{y^2}{2}}dy}^{\frac{1}{2}}
\end{gathered}
\end{equation}
where we have used the fact that for any probability measure $\mu$ one has that
\begin{equation}\nonumber
\int \abs{x}d\mu(x) \leq \pa{\int d\mu(x)}^{\frac{1}{2}}\pa{\int x^2 d\mu(x)}^{\frac{1}{2}}=\pa{\int x^2 d\mu(x)}^{\frac{1}{2}},
\end{equation}
Next, we see that
\begin{equation}\nonumber
\begin{gathered}
\int_{0}^{\infty}\pa{\frac{y}{\sqrt{N}}-1}^2 y^{N-1}e^{-\frac{y^2}{2}}dy = \frac{1}{N}\int_{0}^{\infty} y^{N+1}e^{-\frac{y^2}{2}}dy-\frac{2}{\sqrt{N}}\int_{0}^{\infty} y^{N}e^{-\frac{y^2}{2}}dy + \int_{0}^{\infty} y^{N-1}e^{-\frac{y^2}{2}}dy \\
=\frac{\pa{2\pi}^{\frac{N+2}{2}}}{\abs{\spa{S}{N+1}}N}-\frac{2\pa{2\pi}^{\frac{N+1}{2}}}{\abs{\spa{S}{N}}\sqrt{N}}+\frac{\pa{2\pi}^{\frac{N}{2}}}{\abs{\spa{S}{N-1}}}.
\end{gathered}
\end{equation}
It is easy to check that $\frac{\pa{2\pi}^{\frac{N+2}{2}}}{\abs{\spa{S}{N+1}}N}=\frac{\pa{2\pi}^{\frac{N}{2}}}{\abs{\spa{S}{N-1}}}$ and conclude that
\begin{equation}\nonumber
\begin{gathered}
\abs{\int_\mathbb{R} \phi(x)\Pi_1^{\pa{j}}\pa{F_N}(x)dx - \int_{\mathbb{R}}\phi(v)\Pi_1^{\pa{j}}\pa{\widetilde{F_N}}(v)dv} \\
\leq M_2\pa{\Pi_1^{\pa{j}}\pa{F_N}}^{\frac{1}{2}}\sqrt{2}\pa{1-\frac{\sqrt{2\pi}}{\sqrt{N}}\frac{\abs{\spa{S}{N-1}}}{\abs{\spa{S}{N}}}}^{\frac{1}{2}}
\end{gathered}
\end{equation}
\begin{equation}\nonumber
\begin{gathered}
=\pa{2M_2\pa{\Pi_1^{\pa{j}}\pa{F_N}}}^{\frac{1}{2}}\pa{1-\pa{1+\frac{1}{N}}^{\frac{N}{2}}e^{-\frac{1}{2}}\pa{1+O\pa{\frac{1}{N}}}}^{\frac{1}{2}}
\end{gathered}
\end{equation}
where we have used the surface volume of the sphere formula, and the asymptotic expression for the gamma function. As
\begin{equation}\nonumber
\begin{gathered}
\log\pa{\pa{1+\frac{1}{N}}^{\frac{N}{2}}e^{-\frac{1}{2}}}=\frac{N}{2}\log\pa{1+\frac{1}{N}}-\frac{1}{2} = \frac{N}{2}\pa{\frac{1}{N}-\frac{1}{2N^2} + ...} - \frac{1}{2}\\
=-\frac{1}{4N}\pa{1+O\pa{\frac{1}{N}}},
\end{gathered}
\end{equation}
we find that
\begin{equation}\nonumber
\begin{gathered}
1-\pa{1+\frac{1}{N}}^{\frac{N}{2}}e^{-\frac{1}{2}}\pa{1+O\pa{\frac{1}{N}}}=1-e^{-\frac{1}{4N}\pa{1+O\pa{\frac{1}{N}}}}\pa{1+O\pa{\frac{1}{N}}} = \frac{1}{4N}\pa{1+O\pa{\frac{1}{N}}},
\end{gathered}
\end{equation}
from which we conclude the proof.
\end{proof}
The proof of Lemma \ref{lem: wasserstein 1 distance} shows why we considered the $W_1$ distance. The expression for the marginals of $\widetilde{F_N}$ is complicated, but, as suggested when we investigated its moments, the complexity disappears when one integrates against a simple function. The fact we can replace minimization of a general coupling with integration against Lipschitz functions was the reason for the choice of this metric. However, as mentioned before, we will want to use the HWI inequality in our investigation. For that purpose we will need higher orders of Wasserstein distances. Our next Lemma, which is a simple extension of a result proved by Hauray and Mischler in \cite{HM}, allows us to make the connection between $W_1$ and $W_q$, $q\geq 1$, \emph{as long as we have additional moment control}.
\begin{lemma}\label{lem: hauray and mischeler wasserstein connection}
Let $f,g\in P\pa{\mathbb{R}}$ and let $k>0$. Denote by
\begin{equation}\nonumber
\mathcal{M}_k=\mathcal{M}_k\pa{f}+\mathcal{M}_k\pa{g}=\int_{\mathbb{R}}\pa{1+\abs{v}^2}^{\frac{k}{2}}f(v)dv+\int_{\mathbb{R}}\pa{1+\abs{v}^2}^{\frac{k}{2}}g(v)dv.
\end{equation}
Then, for any $k>q\geq 2$ one has that
\begin{equation}\label{eq: HM W 2 to W 1 connection}
W_q\pa{f,g} \leq 2^{1+\frac{1}{q}}\mathcal{M}_k^{\frac{1}{k}}W_1\pa{f,g}^{\frac{1}{q}-\frac{1}{k}}.
\end{equation}
\end{lemma}
\begin{proof}
Denote by $d(x,y)=\min\pa{\abs{x-y},1}$ and by $\widetilde{W_1}$ the Wasserstien distance of order $1$ associated to $d$. We claim that for all $q\geq 1$ and $R\geq 1$
\begin{equation}\label{eq: HM W q I}
\abs{x-y}^q \leq R^q d(x,y) + \frac{2^k}{R^{k-q}}\pa{\abs{x}^k+\abs{y}^k},
\end{equation}
and leave the proof of this inequality to the Appendix. Integrating (\ref{eq: HM W q I}) against any $\pi\in\Pi\pa{\mu,\nu}$ gives us
\begin{equation}\nonumber
W_q^q\pa{\mu,\nu} \leq R^q \widetilde{W_1}\pa{\mu,\nu} + \frac{2^k}{R^{k-q}}\mathcal{M}_k\pa{\mu,\nu}.
\end{equation}
The choice $R=2\frac{\mathcal{M}^{\frac{1}{k}}_k\pa{\mu,\nu}}{\widetilde{W_1}^{\frac{1}{k}}\pa{\mu,\nu}} \geq 1$ yields
\begin{equation}\nonumber
W_q\pa{\mu,\nu} \leq 2\pa{2}^{\frac{1}{q}}\mathcal{M}^{\frac{1}{k}}_k\pa{\mu,\nu}\widetilde{W_1}^{\frac{1}{q}-\frac{1}{k}}\pa{\mu,\nu},
\end{equation}
from which the result follows as $\widetilde{W_1} \leq W_1$.
\end{proof}
\begin{corollary}\label{cor: wasserstein 2 distance}
Let $F_N\in P\pa{\spa{S}{N-1}\pa{\sqrt{N}}}$ be such that $M_k\pa{\Pi_1^{\pa{j}}\pa{F_N}}<\infty$ for some $k>2$. Then for any $2\leq q <k$
\begin{equation}\label{eq: wasserstein 2 distance}
\begin{gathered}
W_q\pa{\Pi_1^{\pa{j}}\pa{F_N},\Pi_1^{\pa{j}}\pa{\widetilde{F_N}}} \leq 2^{\frac{3}{2}+\frac{1}{q}}\pa{1+\frac{C_k}{2}}^{\frac{1}{k}}\\
\pa{1+M_k\pa{\Pi_1^{\pa{j}}\pa{F_N}}}^{\frac{1}{k}}\frac{M_2\pa{\Pi_1^{\pa{j}}\pa{F_N}}^{\frac{1}{2q}-\frac{1}{2k}}}{\pa{2N}^{\frac{1}{2q}-\frac{1}{2k}}}\pa{1+O\pa{\frac{1}{N}}},
\end{gathered}
\end{equation}
with $C_k=\sup_{N}\pa{\frac{2}{N}}^{\frac{k}{2}}\frac{\Gamma\pa{\frac{N+k}{2}}}{\Gamma\pa{\frac{N}{2}}}$.
\end{corollary}
\begin{proof}
Using the notations of Lemma \ref{lem: hauray and mischeler wasserstein connection} we have that
\begin{equation}\nonumber
\begin{gathered}
\mathcal{M}_k \leq 2^{\frac{k}{2}-1}\pa{2+M_k\pa{\Pi_1^{\pa{j}}\pa{F_N}}+M_k\pa{\Pi_1^{\pa{j}}\pa{\widetilde{F_N}}}} \\
\leq 2^{\frac{k}{2}}\pa{1+\frac{C_k}{2}}\pa{1+M_k\pa{\Pi_1^{\pa{j}}\pa{F_N}}}.
\end{gathered}
\end{equation}
Combining this with Lemma \ref{lem: wasserstein 1 distance} and Lemma \ref{lem: hauray and mischeler wasserstein connection} yields the desired result.
\end{proof}
The next ingredient of the proof that we need is the Fisher Information. While the 'normal' one, defined in Definition \ref{def: relative fisher information} and used in Theorem \ref{thm: improved entropic on sphere with conditions} requires no further discussion, we will require the following lemmas to deal with the Fisher Information on the sphere.  
\begin{lemma}\label{lem: connection between F_j^(N) and Pi_1^{j}(F)N)}
Let $F_N\in P\pa{\kac}$. Then
\begin{equation}\label{eq: connection between F_j and Pi_1(F_N)}
\begin{gathered}
\Pi_1^{\pa{j}}\pa{F_N}(v)=\frac{\abs{\spa{S}{N-2}}}{\abs{\spa{S}{N-1}}\sqrt{N}}\pa{1-\frac{v^2}{N}}^{\frac{N-3}{2}}_+F_j^{\pa{N}}(v) \\
=\frac{1}{\sqrt{2\pi}}\pa{1-\frac{3}{4N}+o\pa{\frac{1}{N}}}\pa{1-\frac{v^2}{N}}^{\frac{N-3}{2}}_+ F_j^{\pa{N}}(v).
\end{gathered}
\end{equation}
\end{lemma}
\begin{proof}
Equality (\ref{eq: connection between F_j and Pi_1(F_N)}) follows immediately from (\ref{eq: normal marginal formula}) with $k=1$ and the fact that
\begin{equation}\nonumber
\begin{gathered}
\frac{\abs{\spa{S}{N-2}}}{\abs{\spa{S}{N-1}}\sqrt{N}}=\frac{1}{\sqrt{2\pi}}\sqrt{\frac{N-1}{N}}\pa{1+\frac{1}{N-1}}^{\frac{N-1}{2}}e^{-\frac{1}{2}}\frac{1+\frac{1}{6N}+\dots}{1+\frac{1}{6(N-1)}+\dots}\\
=\frac{1}{\sqrt{2\pi}}\pa{1-\frac{1}{N}}^{\frac{1}{2}}e^{\pa{-\frac{1}{4(N-1)}+\dots}}\pa{1+\frac{1}{6N}+\dots}\pa{1-\frac{1}{6(N-1)}+\dots}\\
=\frac{1}{\sqrt{2\pi}}\pa{1-\frac{1}{2N}+\dots}\pa{1-\frac{1}{4(N-1)}+\dots}\pa{1-\frac{1}{6N(N-1)}+\dots }\\
=\frac{1}{\sqrt{2\pi}}\pa{1-\frac{3N-2}{4N(N-1)}+\dots}=\frac{1}{\sqrt{2\pi}}\pa{1-\frac{3}{4N}+\dots}.
\end{gathered}
\end{equation}
\end{proof}
\begin{lemma}\label{lem: connection between fisher information on the sphere}
Let $F_N\in P\pa{\kac}$ such that $I_N\pa{F_j^{\pa{N}}}<\infty$. Then
\begin{equation}\label{eq: connection between fisher information on the sphere}
\begin{gathered}
I_N\pa{F_j^{\pa{N}}}=\int_{\R} \pa{1-\frac{v^2}{N}}\abs{\frac{d}{dv} \log\pa{\Pi_1^{\pa{j}}\pa{F_N}(v)}}^2 \Pi_1^{\pa{j}}\pa{F_N}(v) \\
-2\frac{N-3}{N}+\pa{\frac{N-3}{N}}^2 \int_{\R} \frac{v^2 \Pi_1^{\pa{j}}\pa{F_N}(v)}{\pa{1-\frac{v^2}{N}}}dv.
\end{gathered}
\end{equation}
\end{lemma}
\begin{proof}
Denote by $L_{i,j}=\frac{1}{\sqrt{N}}\pa{v_i \partial_j-v_j\partial_i}$. For any $F$ on $\kac$ we have that
\begin{equation}\nonumber 
\begin{gathered}
I_N\pa{F}=\int_{\kac}\frac{\abs{\nabla_{\mathbb{S}}F}^2}{F}d\sigma^N=\frac{1}{2}\sum_{i\not=j}\int_{\kac}\frac{\abs{L_{i,j}F}^2}{F}d\sigma^N \\
=\frac{1}{2}\sum_{i\not=j}\int_{\kac}\abs{L_{i,j}\log F}^2 Fd\sigma^N. 
\end{gathered}
\end{equation}
If $F=f_j$, a function depending only on $v_j$, we find that
\begin{equation}\nonumber
\sum_{i\not=k}\abs{L_{i,k}f_j}^2 = \frac{2}{N}\sum_{i\not=j} v_i^2 \pa{\frac{d}{dv_j}f_j}^2 =2\pa{1-\frac{v_j^2}{N}}\pa{\frac{d}{dv_j}f_j }^2.
\end{equation}
Thus, using (\ref{eq: connection between F_j and Pi_1(F_N)}) we find that
\begin{equation}\nonumber
\begin{gathered}
I_N\pa{F_j^{\pa{N}}}= \int_{\kac}\pa{1-\frac{v^2}{N}}\abs{\frac{d}{dv}\log\pa{\frac{\Pi_1^{\pa{j}}\pa{F_N}(v)}{\frac{\abs{\spa{S}{N-2}}}{\abs{\spa{S}{N-1}}\sqrt{N}}\pa{1-\frac{v^2}{N}}^{\frac{N-3}{2}}}}}^2 F^{\pa{N}}_j(v)d\sigma^N \\
=\int_{\R}\pa{1-\frac{v^2}{N}}\abs{\frac{d}{dv}\log\pa{\Pi_1^{\pa{j}}\pa{F_N}(v)}-\frac{N-3}{2}\frac{d}{dv}\log\pa{1-\frac{v^2}{N}}}^2 \Pi_1^{\pa{j}}\pa{F_N}(v)dv \\
\end{gathered}
\end{equation}
\begin{equation}\nonumber
\begin{gathered}
=\int_{\R}\pa{1-\frac{v^2}{N}}\abs{\frac{d}{dv}\log\pa{\Pi_1^{\pa{j}}\pa{F_N}(v)}+\frac{(N-3)v}{N\pa{1-\frac{v^2}{N}}}}^2 \Pi_1^{\pa{j}}\pa{F_N}(v)dv\\ 
=\int_{\R}\pa{1-\frac{v^2}{N}}\abs{\frac{d}{dv}\log\pa{\Pi_1^{\pa{j}}\pa{F_N}(v)}}^2 \Pi_1^{\pa{j}}\pa{F_N}(v)dv \\
\end{gathered}
\end{equation}
\begin{equation}\nonumber
\begin{gathered}
+2\frac{N-3}{N}\int_{\R} v\frac{\frac{d}{dv}\Pi_1^{\pa{j}}\pa{F_N}(v)}{\Pi_1^{\pa{j}}\pa{F_N}(v)}\Pi_1^{\pa{j}}\pa{F_N}(v)dv
+\pa{\frac{N-3}{N}}^2 \int_{\R}\frac{v^2\Pi_1^{\pa{j}}\pa{F_N}(v)}{1-\frac{v^2}{N}}dv,
\end{gathered}
\end{equation}
where we have used (\ref{eq: sphere integration}) in the second line. Since 
\begin{equation}\nonumber
\int_{\R} v \frac{d}{dv}\Pi_1^{\pa{j}}\pa{F_N}(v)dv= -\int_{\R}\Pi_1^{\pa{j}}\pa{F_N}(v)dv=-1,
\end{equation}
we obtain the desired result.
\end{proof}
Using our acquired knowledge till this point we can now find a quantitative estimation in the difference of the entropies of the marginals and the marginals of the extension.
\begin{theorem}\label{thm: connection between entropies of Pi_1(F_N) and Pi_1(tilde{F_N})}
Let $F_N\in P\pa{\kac}$ such that $M_k\pa{\Pi_1^{\pa{j}}\pa{F_N}}<\infty$ for some $k>2$.\\
$(i)$ If $I\pa{\Pi_1^{\pa{j}}\pa{F_N}}<\infty$ then there exists $C_2>0$, independent of $N$ and $F_N$, such that
\begin{equation}\label{eq: connection between entropies of Pi_1(F_N) and Pi_1(tilde{F_N}) I}
\begin{gathered}
H\pa{\Pi_1^{\pa{j}}\pa{F_N} | \gamma} \leq H\pa{\Pi_1^{\pa{j}}\pa{\widetilde{F_N}} | \gamma} 
+4C_2\pa{1+\frac{C_k}{2}}^{\frac{1}{k}} \pa{1+M_k\pa{\Pi_1^{\pa{j}}\pa{F_N}}}^{\frac{1}{k}}\\
\frac{\pa{I\pa{\Pi_1^{\pa{j}}\pa{F_N}}+M_2\pa{\Pi_1^{\pa{j}}\pa{F_N}}-2}^{\frac{1}{2}} M_2\pa{\Pi_1^{\pa{j}}\pa{F_N}}^{\frac{1}{4}-\frac{1}{2k}}}{\pa{2N}^{\frac{1}{4}-\frac{1}{2k}}},
\end{gathered}
\end{equation}
where $C_k=\sup_{N}\pa{\frac{2}{N}}^{\frac{k}{2}}\frac{\Gamma\pa{\frac{N+k}{2}}}{\Gamma\pa{\frac{N}{2}}}$.\\
$(ii)$ If $I_N\pa{\Pi_1^{\pa{j}}\pa{F_N}}<\infty$ and there exists $2 <q<k$ such that
\begin{equation}\label{eq: P_q condition}
P_q^{\pa{j}}\pa{F_N}=\int_{\R}\frac{\Pi_1^{\pa{j}}\pa{F_N}(v)}{\pa{1-\frac{v^2}{N}}^{\frac{q}{q-2}}}dv < \infty
\end{equation}
then there exists $C_2>0$, independent of $N$ and $F_N$, such that
\begin{equation}\label{eq: connection between entropies of Pi_1(F_N) and Pi_1(tilde{F_N}) II}
\begin{gathered}
H\pa{\Pi_1^{\pa{j}}\pa{F_N} | \gamma} \leq H\pa{\Pi_1^{\pa{j}}\pa{\widetilde{F_N}} | \gamma} 
+2^{\frac{3}{2}+\frac{2}{q}}C_2\pa{1+\frac{C_k}{2}}^{\frac{1}{k}}\\
\pa{\pa{I_N\pa{F_j^{\pa{N}}}+2}^{\frac{q}{2(q-1)}}\pa{P_q^{\pa{j}}}^{\frac{q-2}{2(q-1)}}+1+M_2\pa{\Pi_1^{\pa{j}}\pa{F_N}}}^{\frac{q-1}{q}} \\
\pa{1+M_k\pa{\Pi_1^{\pa{j}}\pa{F_N}}}^{\frac{1}{k}}\frac{M_2\pa{\Pi_1^{\pa{j}}\pa{F_N}}^{\frac{1}{2q}-\frac{1}{2k}}}{\pa{2N}^{\frac{1}{2q}-\frac{1}{2k}}},
\end{gathered}
\end{equation}
where $C_k=\sup_{N}\pa{\frac{2}{N}}^{\frac{k}{2}}\frac{\Gamma\pa{\frac{N+k}{2}}}{\Gamma\pa{\frac{N}{2}}}$.
\end{theorem}
\begin{proof}
$(i)$ The HWI inequality states that 
\begin{equation}\nonumber
H(f|\gamma) \leq H(g | \gamma) + \sqrt{I\pa{f|\gamma}}W_2\pa{f,g}.
\end{equation}
Together with the simple identity for $f\in P\pa{\R}$
\begin{equation}\nonumber
I\pa{f|\gamma}=I(f) +\int_\mathbb{R} v^2 f(v)dv -2,
\end{equation}
Corollary \ref{cor: wasserstein 2 distance} with $q=2$, and the fact that the $O\pa{\frac{1}{N}}$ term in (\ref{eq: wasserstein 2 distance}) was independent in $F_N$ we conclude the result.\\
$(ii)$ This part of the theorem requires a slight modification of the HWI inequality. Following the proof of the inequality, see for instance \cite{C-E,Vtransport2}, one notices that replacing the Cauchy-Schwartz inequality with the H\"older inequality (and using the uniqueness of the transportation map if needed) gives us that for any $1<p<\infty$
\begin{equation}\nonumber
H(f|\gamma) \leq H(g | \gamma) + \pa{\int_{\R} \abs{\frac{d}{dv}\log \pa{\frac{f(v)}{\gamma(v)}}}^p f(v)dv}^{\frac{1}{p}}W_q\pa{f,g}.
\end{equation}
where $q$ is the H\"older conjugate of $p$.  For $1\leq p <2$ we find that
\begin{equation}\nonumber
\begin{gathered}
\int_{\R} \abs{\frac{d}{dv}\log \pa{\frac{f(v)}{\gamma(v)}}}^pf(v)dv=\int_{\R} \abs{\frac{d}{dv}\log f(v) + v}^pf(v)dv \\
\leq 2^{p-1}\pa{\int_{\R} \abs{\frac{d}{dv}\log f(v)}^pf(v)dv+\int_{\R} \abs{v}^pf(v)dv} \leq 2^{p-1}\pa{\int_{\R} \abs{\frac{d}{dv}\log f(v)}^pf(v)dv+1+M_2(f)},
\end{gathered}
\end{equation}
and if in addition $f\in P\pa{\R}$ is supported in $[-\sqrt{N},\sqrt{N}]$ then
\begin{equation}\nonumber
\begin{gathered}
\int_{\R} \abs{\frac{d}{dv}\log f(v)}^pf(v)dv \leq \pa{\int_{\R} \pa{1-\frac{v^2}{N}}\abs{\frac{d}{dv}\log f(v)}^2 f(v)dv}^{\frac{p}{2}}\pa{\int_{\R} \frac{f(v)}{\pa{1-\frac{v^2}{N}}^{\frac{p}{2-p}}}dv}^{\frac{2-p}{2}}.
\end{gathered}
\end{equation}
We conclude that for $p=\frac{q}{q-1}$, where $q$ is as in (\ref{eq: P_q condition}), one has that
\begin{equation}\nonumber
\begin{gathered}
H\pa{\Pi_1^{\pa{j}}\pa{F_N} | \gamma} \leq H\pa{\Pi_1^{\pa{j}}\pa{\widetilde{F_N}} | \gamma} \\
+ 2^{\frac{1}{q}}\pa{\pa{\int_{\R} \pa{1-\frac{v^2}{N}}\abs{\frac{d}{dv}\log \pa{\Pi_1^{\pa{j}}\pa{F_N}(v)}}^2 \Pi_1^{\pa{j}}\pa{F_N}(v)dv}^{\frac{q}{2(q-1)}}\pa{P_q^{\pa{j}}\pa{F_N}}^{\frac{q-2}{2(q-1)}}+1+M_2\pa{\Pi_1^{\pa{j}}\pa{F_N}(v)}}^{\frac{q-1}{q}}\\
W_q\pa{\Pi_1^{\pa{j}}\pa{F_N},\Pi_1^{\pa{j}}\pa{\widetilde{F_N}}}.
\end{gathered}
\end{equation}
The result follows from the inequality
\begin{equation}\nonumber
\int_{\R}\pa{1-\frac{v^2}{N}}\abs{\frac{d}{dv} \log\pa{\Pi_1^{\pa{j}}\pa{F_N}(v)}}^2 \Pi_1^{\pa{j}}\pa{F_N}(v)dv \leq I_N\pa{F_j^{\pa{N}}}+2\frac{N-3}{N},
\end{equation}
which is a consequence of Lemma \ref{lem: connection between fisher information on the sphere}, and Corollary \ref{cor: wasserstein 2 distance}.
\end{proof}
\section{The Entropy Relation - From Marginals On The Sphere to Marginals On The Line.}\label{sec: the entropy relation II}
In Section \ref{sec: entropy relation I} we have seen how to relate the relative entropy of $\Pi_1^{\pa{j}}\pa{F_N}$ to that of $\Pi_1^{\pa{j}}\pa{\widetilde{F_N}}$, gaining a quantitative estimation on the difference between the two. However, our entropic inequalities, (\ref{eq: improved entropic on sphere with conditions}) and (\ref{eq: imrpoved entropic on sphere with spherical fisher information}), relate to the entropy of the marginal on the sphere. In this section we will explore the connection between the entropies of the marginals on the sphere and those of the marginals.
\begin{lemma}\label{lem: connection between entropies of F_j and Pi_1(F_N)}
Let $F_N\in P\pa{\kac}$. Then
\begin{equation}\label{eq: entropic connection between entropies of F_j and Pi_1(F_N)}
\begin{gathered}
\int_{\kac}F_j^{\pa{N}}\log F_J^{\pa{N}}d\sigma^N = H\pa{\Pi_1^{\pa{j}}\pa{F_N} | \gamma}-\log\pa{1-\frac{3}{4N}+o\pa{\frac{1}{N}}} \\
-\frac{1}{2}\int_{-\sqrt{N}}^{\sqrt{N}}v^2\Pi_1^{\pa{j}}\pa{F_N}(v)dv - \frac{N-3}{2}\int_{-\sqrt{N}}^{\sqrt{N}}\Pi_1^{\pa{j}}\pa{F_N}(v)\log\pa{1-\frac{v^2}{N}}dv.
\end{gathered}
\end{equation}
\end{lemma}
\begin{proof}

Using (\ref{eq: sphere integration}) we find that
\begin{equation}\nonumber
\begin{gathered}
\int_{\kac}F_j^{\pa{N}}\log F_J^{\pa{N}}d\sigma^N  = \frac{\abs{\spa{S}{N-2}}}{\abs{\spa{S}{N-1}}\sqrt{N}}\int_{-\sqrt{N}}^{\sqrt{N}}\pa{1-\frac{v^2}{N}}^{\frac{N-3}{2}}_+ F_j^{\pa{N}}(v)\log \pa{F_J^{\pa{N}}(v)}dv\\
=\int_{-\sqrt{N}}^{\sqrt{N}}\Pi_1^{\pa{j}}\pa{F_N}(v)\log\pa{\Pi_1^{\pa{j}}\pa{F_N}(v)}dv - \int_{-\sqrt{N}}^{\sqrt{N}}\Pi_1^{\pa{j}}\pa{F_N}(v)\log\pa{\frac{\abs{\spa{S}{N-2}}}{\abs{\spa{S}{N-1}}\sqrt{N}}\pa{1-\frac{v^2}{N}}^{\frac{N-3}{2}}}dv\\
\end{gathered}
\end{equation}
\begin{equation}\nonumber
\begin{gathered}
=H\pa{\Pi_1^{\pa{j}}\pa{F_N}|\gamma}+\log\pa{\frac{1}{\sqrt{2\pi}}}-\frac{1}{2}\int_{-\sqrt{N}}^{\sqrt{N}}v^2\Pi_1^{\pa{j}}\pa{F_N}(v)dv-\log\pa{\frac{\abs{\spa{S}{N-2}}}{\abs{\spa{S}{N-1}}\sqrt{N}}}\\
-\frac{N-3}{2}\int_{-\sqrt{N}}^{\sqrt{N}}\Pi_1^{\pa{j}}\pa{F_N}(v)\log\pa{1-\frac{v^2}{N}}dv, 
\end{gathered}
\end{equation}
yielding the desired result.
\end{proof}
\begin{lemma}\label{lem: correction terms in entropic relationship}
Let $F_N\in P\pa{\kac}$ such that $M_k\pa{\Pi_1^{\pa{j}}\pa{F_N}}<\infty$ for some $k>2$.\\
$(i)$ If $I\pa{\Pi_1^{\pa{j}}\pa{F_N}}<\infty$ then for any sequence $0<\epsilon_N<1$, converging to zero, we have that
\begin{equation}\label{eq: correction terms in entropic relationship}
\begin{gathered}
-\frac{1}{2}\int_{-\sqrt{N}}^{\sqrt{N}}v^2\Pi_1^{\pa{j}}\pa{F_N}(v)dv - \frac{N-3}{2}\int_{-\sqrt{N}}^{\sqrt{N}}\Pi_1^{\pa{j}}\pa{F_N}(v)\log\pa{1-\frac{v^2}{N}}dv \\
\leq \frac{M_k\pa{\Pi_1^{\pa{j}}\pa{F_N}}}{2N^{\frac{k}{2}-1}\epsilon_N} + \frac{I\pa{\Pi_1^{\pa{j}}\pa{F_N}}^{\frac{p-1}{2p}} M_k\pa{\Pi_1^{\pa{j}}\pa{F_N}}^{\frac{1}{p}}C_p}{2\pa{1-\epsilon_N}^{\frac{k}{2p}}N^{\frac{1}{2}\pa{\frac{k+1}{p}-3}}}
\end{gathered}
\end{equation}
where $1<p<\frac{k}{2}$ and
\begin{equation}\label{eq: def of C_p}
C_p=\pa{\int_{\abs{x}<1}\abs{\log\pa{1-x^2}}^{\frac{p}{p-1}}dx}^{\frac{p-1}{p}}.
\end{equation}
$(ii)$ If $I_N\pa{F^{\pa{N}}_j}<\infty$ then for any sequence $0<\epsilon_N<1$, converging to zero, we have that
\begin{equation}\label{eq: correction terms in entropic relationship fisher sphere}
\begin{gathered}
-\frac{1}{2}\int_{-\sqrt{N}}^{\sqrt{N}}v^2\Pi_1^{\pa{j}}\pa{F_N}(v)dv - \frac{N-3}{2}\int_{-\sqrt{N}}^{\sqrt{N}}\Pi_1^{\pa{j}}\pa{F_N}(v)\log\pa{1-\frac{v^2}{N}}dv 
\leq \frac{M_k\pa{\Pi_1^{\pa{j}}\pa{F_N}}}{2N^{\frac{k}{2}-1}\epsilon_N}\\
 + \frac{N}{2(N-3)(1-\epsilon_N)^{\frac{k}{4}+\frac{1}{2}}}\pa{I_N\pa{F_j^{\pa{N}}}+2\frac{N-3}{N}}^{\frac{1}{2}}\frac{l_N}{N^{\frac{k}{4}-\frac{1}{2}}}M_k\pa{\Pi_1^{\pa{j}}\pa{F_N}}^{\frac{1}{2}},
\end{gathered}
\end{equation}
where $l_N=\sqrt{\sup_{x\in[0,\epsilon_N]}x\pa{\log x}^2}$
\end{lemma}
\begin{proof}
Using the inequality
\begin{equation}\nonumber
-\log\pa{1-x}<\frac{x}{1-x}
\end{equation}
for $0<x<1$, we find that
\begin{equation}\nonumber
\begin{gathered}
-\frac{N-3}{2}\log\pa{1-\frac{v^2}{N}}-\frac{v^2}{2} <\frac{N-3}{2}\frac{v^2}{N-v^2} - \frac{v^2}{2} \\
=\frac{(N-3)v^2-(N-v^2)v^2}{2(N-v^2)}=\frac{v^4-3v^2}{2(N-v^2)} <\frac{v^4}{2(N-v^2)}.
\end{gathered}
\end{equation}
 For any $R>0$ we have that
\begin{equation}\nonumber
\begin{gathered}
\int_{\abs{v}<R}\pa{-\frac{N-3}{2}\log\pa{1-\frac{v^2}{N}}-\frac{v^2}{2}}\Pi_1^{\pa{j}}\pa{F_N}(v)dv \\
\leq \frac{1}{2(N-R^2)}\int_{\abs{v}<R}v^4\Pi_1^{\pa{j}}\pa{F_N}(v)dv \leq \frac{R^{4-k}}{2(N-R^2)}M_k\pa{\Pi_1^{\pa{j}}\pa{F_N}}.
\end{gathered}
\end{equation}
Picking $R=\sqrt{N\pa{1-\epsilon_N}}$, with $0<\epsilon_N<1$ going to zero, we find that
\begin{equation}\label{eq: correction terms I}
\begin{gathered}
\int_{\abs{v}<\sqrt{N\pa{1-\epsilon_N}}}\pa{-\frac{N-3}{2}\log\pa{1-\frac{v^2}{N}}-\frac{v^2}{2}}\Pi_1^{\pa{j}}\pa{F_N}(v)dv \\ 
\leq \frac{N^{1-\frac{k}{2}}}{2\epsilon_N}M_k\pa{\Pi_1^{\pa{j}}\pa{F_N}}=\frac{M_k\pa{\Pi_1^{\pa{j}}\pa{F_N}}}{2N^{\frac{k}{2}-1}\epsilon_N}.
\end{gathered}
\end{equation}
The difference between $(i)$ and $(ii)$ manifests itself in the domain $\abs{v}\geq \sqrt{N\pa{1-\epsilon_N}}$. To prove $(i)$ we notice that
\begin{equation}\nonumber
\begin{gathered}
-\frac{1}{2}\int_{\abs{v}\geq \sqrt{N\pa{1-\epsilon_N}}}v^2\Pi_1^{\pa{j}}\pa{F_N}(v)dv - \frac{N-3}{2}\int_{\abs{v}\geq \sqrt{N\pa{1-\epsilon_N}}}\Pi_1^{\pa{j}}\pa{F_N}(v)\log\pa{1-\frac{v^2}{N}}dv \\
\leq -\frac{1}{2}\int_{\abs{v}\geq \sqrt{N\pa{1-\epsilon_N}}}N\Pi_1^{\pa{j}}\pa{F_N}(v)\log\pa{1-\frac{v^2}{N}}dv\\
\end{gathered}
\end{equation}
\begin{equation}\label{eq: correction terms II}
\begin{gathered}
\leq \frac{1}{2\pa{1-\epsilon_N}}\int_{\abs{v}\geq \sqrt{N\pa{1-\epsilon_N}}}v^2\Pi_1^{\pa{j}}\pa{F_N}(v)\pa{-\log\pa{1-\frac{v^2}{N}}}dv \\
\leq \frac{1}{2\pa{1-\epsilon_N}}\pa{\int_{\abs{v}\geq \sqrt{N\pa{1-\epsilon_N}}}\abs{v}^{2p}\Pi_1^{\pa{j}}\pa{F_N}(v)dv}^{\frac{1}{p}}\\
\end{gathered}
\end{equation}
\begin{equation}\nonumber
\begin{gathered}
\pa{\int_{-\sqrt{N}}^{\sqrt{N}}\abs{\log\pa{1-\frac{v^2}{N}}}^{\frac{p}{p-1}}\Pi_1^{\pa{j}}\pa{F_N}(v)dv}^{\frac{p-1}{p}} \\
\leq \frac{1}{2\pa{1-\epsilon_N}}\pa{\frac{1}{\pa{N\pa{1-\epsilon_N}}^{\frac{k}{2}-p}}\int_{\abs{v}\geq \sqrt{N\pa{1-\epsilon_N}}}\abs{v}^{k}\Pi_1^{\pa{j}}\pa{F_N}(v)dv}^{\frac{1}{p}} \\
\end{gathered}
\end{equation}
\begin{equation}\nonumber
\begin{gathered}
\Norm{\Pi_1^{\pa{j}}\pa{F_N}}_{\infty}^{\frac{p-1}{p}}\pa{\int_{-\sqrt{N}}^{\sqrt{N}}\abs{\log\pa{1-\frac{v^2}{N}}}^{\frac{p}{p-1}}dv}^{\frac{p-1}{p}} \\
\leq \frac{\Norm{\Pi_1^{\pa{j}}\pa{F_N}}_{\infty}^{\frac{p-1}{p}} M_k\pa{\Pi_1^{\pa{j}}\pa{F_N}}^{\frac{1}{p}}}{2\pa{1-\epsilon_N}^{\frac{k}{2p}}N^{\frac{k}{2p}-1}}
N^{\frac{p-1}{2p}}\pa{\int_{\abs{x}<1}\abs{\log\pa{1-x^2}}^{\frac{p}{p-1}}dx}^{\frac{p-1}{p}}\\
=\frac{\Norm{\Pi_1^{\pa{j}}\pa{F_N}}_{\infty}^{\frac{p-1}{p}} M_k\pa{\Pi_1^{\pa{j}}\pa{F_N}}^{\frac{1}{p}}C_p}{2\pa{1-\epsilon_N}^{\frac{k}{2p}}N^{\frac{1}{2}\pa{\frac{k+1}{p}-3}}},
\end{gathered}
\end{equation}
where $p>1$ was chosen such that $p<\frac{k}{2}$. The result follows from (\ref{eq: correction terms I}), (\ref{eq: correction terms II}) and the following inequality: For any $f\in P\pa{\mathbb{R}}$ with a finite Fisher Information $I(f)$ one has that
\begin{equation}\nonumber
\Norm{f}_{\infty} \leq \pa{I\pa{f}}^{\frac{1}{2}}.
\end{equation}
In order to prove $(ii)$ we notice that
\begin{equation}\nonumber
\begin{gathered}
-\frac{1}{2}\int_{\abs{v}\geq \sqrt{N\pa{1-\epsilon_N}}}v^2\Pi_1^{\pa{j}}\pa{F_N}(v)dv - \frac{N-3}{2}\int_{\abs{v}\geq \sqrt{N\pa{1-\epsilon_N}}}\Pi_1^{\pa{j}}\pa{F_N}(v)\log\pa{1-\frac{v^2}{N}}dv \\
\leq -\frac{1}{2}\int_{\abs{v}\geq \sqrt{N\pa{1-\epsilon_N}}}N\Pi_1^{\pa{j}}\pa{F_N}(v)\log\pa{1-\frac{v^2}{N}}dv \\
\leq \frac{1}{2(1-\epsilon_N)} \int_{\abs{v}\geq \sqrt{N\pa{1-\epsilon_N}}}v^2\Pi_1^{\pa{j}}\pa{F_N}(v)\abs{\log\pa{1-\frac{v^2}{N}}}dv
\end{gathered}
\end{equation}
\begin{equation}\nonumber
\begin{gathered}
\leq \frac{1}{2(1-\epsilon_N)}\pa{\int_{\abs{v}\geq \sqrt{N\pa{1-\epsilon_N}}}\frac{v^2 \Pi_1^{\pa{j}}\pa{F_N}(v)}{1-\frac{v^2}{N}}}^{\frac{1}{2}}
\pa{\int_{\abs{v}\geq \sqrt{N\pa{1-\epsilon_N}}}v^2\pa{1-\frac{v^2}{N}}\abs{\log\pa{1-\frac{v^2}{N}}}^2\Pi_1^{\pa{j}}\pa{F_N}(v)dv}^{\frac{1}{2}} \\
\leq \frac{N}{2(N-3)(1-\epsilon_N)}\pa{I_N\pa{F_j^{\pa{N}}}+2\frac{N-3}{N}}^{\frac{1}{2}}\frac{l_N}{\pa{N(1-\epsilon_N}^{\frac{k}{4}-\frac{1}{2}}}M_k\pa{\Pi_1^{\pa{j}}\pa{F_N}}^{\frac{1}{2}},
\end{gathered}
\end{equation}
showing the result.
\end{proof}
Combining Lemma \ref{lem: connection between entropies of F_j and Pi_1(F_N)} and \ref{lem: correction terms in entropic relationship} with the choice $\epsilon_N=N^{-\beta}$ gives us
\begin{theorem}\label{thm: entropy of marginals is close}
Let $F_N\in P\pa{\kac}$ such that $M_k\pa{\Pi_1^{\pa{j}}\pa{F_N}}<\infty$ for some $k>2$.\\
$(i)$ If $I\pa{\Pi_1^{\pa{j}}\pa{F_N}}<\infty$ then there exists $C_1>0$, independent of $N$ and $F_N$, such that for any $\beta>0$ and any $1<p<\min\pa{\frac{k+1}{3},\frac{k}{2}}$
\begin{equation}\label{eq:  entropy of marginals is close}
\begin{gathered}
\int_{\kac}F_j^{\pa{N}}\log F_j^{\pa{N}} d\sigma^N \leq H\pa{\Pi_1^{\pa{j}}\pa{F_N} | \gamma} +\frac{C_1}{N} \\
+\frac{M_k\pa{\Pi_1^{\pa{j}}\pa{F_N}}}{2N^{\frac{k}{2}-1-\beta}} + \frac{I\pa{\Pi_1^{\pa{j}}\pa{F_N}}^{\frac{p-1}{2p}} M_k\pa{\Pi_1^{\pa{j}}\pa{F_N}}^{\frac{1}{p}}C_p}{2\pa{1-\frac{1}{N^{\beta}}}^{\frac{k}{2p}}N^{\frac{1}{2}\pa{\frac{k+1}{p}-3}}},
\end{gathered}
\end{equation}
where $C_p=\pa{\int_{\abs{x}<1}\abs{\log\pa{1-x^2}}^{\frac{p}{p-1}}dx}^{\frac{p-1}{p}}$. \\
$(ii)$ If $I_N\pa{F^{\pa{N}}_j}<\infty$ then there exists $C_1>0$, independent of $N$ and $F_N$, such that for any $\beta>0$
\begin{equation}\label{eq: entropy of marginals is close fisher sphere}
\begin{gathered}
\int_{\kac}F_j^{\pa{N}}\log F_j^{\pa{N}} d\sigma^N \leq H\pa{\Pi_1^{\pa{j}}\pa{F_N} | \gamma} +\frac{C_1}{N} \\
\leq \frac{M_k\pa{\Pi_1^{\pa{j}}\pa{F_N}}}{2N^{\frac{k}{2}-1-\beta}}
 + \frac{N}{2(N-3)(1-\frac{1}{N^\beta})^{\frac{k}{4}+\frac{1}{2}}}\pa{I_N\pa{F_j^{\pa{N}}}+2\frac{N-3}{N}}^{\frac{1}{2}}\frac{\eta_{N,\beta}}{N^{\frac{k}{4}-\frac{1}{2}}}M_k\pa{\Pi_1^{\pa{j}}\pa{F_N}}^{\frac{1}{2}},
\end{gathered}
\end{equation}
where $\eta_{N,\beta}=\sqrt{\sup_{x\in\left[0,N^{-\beta} \right]}x\pa{\log x}^2}$
\end{theorem}
We now have all the tools to prove our main theorems.

\section{Proof of the Main Theorems.}\label{sec: proofs}
In the previous couple of sections we have managed to find conditions on our original probability density, $F_N$, such that the appropriate marginals on the sphere, marginals on the line and the marginals of the extension give close values for the appropriate entropy functional. Combining all these result will lead to the proof of our main theorems, which is the subject of this section.\\
We begin with a simple technical lemma, whose proof we leave to the Appendix:
\begin{lemma}\label{lem: holder type inequality}
Let $\br{a_{j,i}}_{j=1,\dots,m \; i=1,\dots,N}$ be non-negative numbers. Let $p_1,\dots,p_m$ be positive numbers such that $\sum_{j=1}^m \frac{1}{p_j}\leq 1$. Then
\begin{equation}\label{eq: holder type inequality}
\sum_{i=1}^N \pa{\Pi_{j=1}^m a_{j,i}^{\frac{1}{p_j}}} \leq \Pi_{j=1}^m \pa{\frac{\sum_{i=1}^N a_{j,i}}{N}}^{\frac{1}{p_j}}N.
\end{equation}
\end{lemma}

\begin{theorem}\label{thm: improved entropic on sphere general I}
Let $F_N\in P\pa{\kac}$ such that there exists $k>2$ with
\begin{equation}\nonumber
\mathcal{A}^M_{N,k}=\frac{\sum_{j=1}^N M_k \pa{\Pi_1^{\pa{j}}\pa{F_N}}}{N}<\infty.
\end{equation}
Assume in addition that
\begin{equation}\nonumber
\mathcal{A}^I_N = \frac{\sum_{i=1}^N I\pa{\Pi_1^{\pa{j}}\pa{F_N}}}{N}<\infty.
\end{equation}
Then there exists $C_1,C_2>0$ independent of $N$ and $F_N$, such that for any $\beta>0$ and $1<p<\min\pa{\frac{k+1}{3},\frac{k}{2}}$
\begin{equation}\label{eq: improved entropic on sphere general I}
\begin{gathered}
\sum_{j=1}^N \int_{\kac}F_J^{\pa{N}}\log F_j^{\pa{N}}d\sigma^N \leq H_N\pa{F_N} +C_1\\
 +\pa{\frac{4C_2\pa{1+\frac{C_k}{2}}^{\frac{1}{k}}}{\pa{2N}^{\frac{1}{4}-\frac{1}{2k}}}\pa{\mathcal{A}^{I}_{N} -1 }^{\frac{1}{2}}
\pa{1+\mathcal{A}^M_{N,k}}^{\frac{1}{k}}} N\\
 +\pa{\frac{\mathcal{A}^M_{N,k}}{2N^{\frac{k}{2}-1-\beta}} }N 
+\pa{\frac{C_p \pa{\mathcal{A}^I_N}^{\frac{p-1}{2p}} \pa{\mathcal{A}^M_{N,k}}^{\frac{1}{p}}}{2\pa{1-\frac{1}{N^{\beta}}}^{\frac{k}{2p}}N^{\frac{1}{2}\pa{\frac{k+1}{p}-3}}}}N,
\end{gathered}
\end{equation}
where $C_p=\pa{\int_{\abs{x}<1}\abs{\log\pa{1-x^2}}^{\frac{p}{p-1}}}^{\frac{p-1}{p}}$ and $C_k=\sup_{N}\pa{\frac{2}{N}}^{\frac{k}{2}}\frac{\Gamma\pa{\frac{N+k}{2}}}{\Gamma\pa{\frac{N}{2}}}$.
\end{theorem}
\begin{proof}
This follows immediately from Theorem \ref{thm: connection between entropies of Pi_1(F_N) and Pi_1(tilde{F_N})}, Theorem \ref{thm: entropy of marginals is close}, Lemma \ref{lem: holder type inequality}, the fact that for any $F_N\in P\pa{\kac}$
\begin{equation}\nonumber
\begin{gathered}
\sum_{j=1}^N M_2\pa{\Pi_1^{\pa{j}}\pa{F_N}}=\sum_{j=1}^N \int_{\kac} v_j^2 F_N d\sigma^N = N,
\end{gathered}
\end{equation}
and inequality (\ref{eq: entropic inequality on R}) applied to $\widetilde{F_N}$ together with 
\begin{equation}\nonumber
H\pa{\widetilde{F_N} | \gamma_N}=H_N\pa{F_N},
\end{equation}
proven in Lemma \ref{lem: consistency}.
\end{proof}
\begin{theorem}\label{thm: improved entropic on sphere general II}
Let $F_N\in P\pa{\kac}$ such that there exists $k>2$ with
\begin{equation}\nonumber
\mathcal{A}^M_{N,k}=\frac{\sum_{j=1}^N M_k \pa{\Pi_1^{\pa{j}}\pa{F_N}}}{N}<\infty.
\end{equation}
Assume in addition that 
\begin{equation}\nonumber
\mathcal{A}^{I_\mathbb{S}}_N = \frac{\sum_{i=1}^N I_N\pa{\Pi_1^{\pa{j}}\pa{F_N}}}{N}<\infty.
\end{equation}
and that there exists $2<q<k$ such that
\begin{equation}\nonumber
\mathcal{A}^P_{N,q}=\frac{\sum_{j=1}^N P_{q}^{\pa{j}}\pa{F_N}}{N}<\infty.
\end{equation}
where 
\begin{equation}\nonumber
P_{q}^{\pa{j}}\pa{F_N}=\int \frac{\Pi_1^{\pa{j}}\pa{F_N}(v)}{\pa{1-\frac{v^2}{N}}^{\frac{q}{q-2}}}dv.
\end{equation}
Then there exists $C_1,C_2>0$ independent of $N$ and $F_N$, such that for any $\beta>0$
\begin{equation}\label{eq: improved entropic on sphere general II}
\begin{gathered}
\sum_{j=1}^N \int_{\kac}F_J^{\pa{N}}\log F_j^{\pa{N}}d\sigma^N \leq H_N\pa{F_N} +C_1\\
+C_2 2^{\frac{3}{2}+\frac{2}{q}}\pa{1+\frac{C_k}{2}}^{\frac{1}{k}}\pa{\pa{\mathcal{A}^{I_{\mathbb{S}}}_N+2}^{\frac{q}{2(q-1)}}\pa{\mathcal{A}^P_{N,q}}^{\frac{q-2}{2(q-1)}}+2}^{\frac{q}{q-1}}
\frac{\pa{1+\mathcal{A}^M_{N,k}}^{\frac{1}{k}}}{\pa{2N}^{\frac{1}{2q}-\frac{1}{2k}}}N
\\
 +\pa{\frac{\mathcal{A}^M_{N,k}}{2N^{\frac{k}{2}-1-\beta}} }N 
+\frac{N}{2(N-3)}\frac{\eta_{N,\beta}}{N^{\frac{k}{4}-\frac{1}{2}}\pa{1-\frac{1}{N^\beta}}^{\frac{k}{4}+\frac{1}{2}}}\pa{\mathcal{A}^{I_{\mathbb{S}}}_N+2}^{\frac{1}{2}}\pa{\mathcal{A}^M_{N,k}}^{\frac{1}{2}}N,
\end{gathered}
\end{equation}
where $C_k=\sup_{N}\pa{\frac{2}{N}}^{\frac{k}{2}}\frac{\Gamma\pa{\frac{N+k}{2}}}{\Gamma\pa{\frac{N}{2}}}$ and $\eta_{N\beta}=\sup_{x\in \left[0,N^{-\beta} \right]}x\pa{\log x}^2$.
\end{theorem}
\begin{proof}
Much like the proof of Theorem \ref{thm: improved entropic on sphere general I}, we just rely on Theorem \ref{thm: connection between entropies of Pi_1(F_N) and Pi_1(tilde{F_N})}, Theorem \ref{thm: entropy of marginals is close}, Lemma \ref{lem: holder type inequality}, the simple second moment computation and the entropic inequality for $\widetilde{F_N}$.
\end{proof}
\begin{proof}[Proof of Theorem \ref{thm: improved entropic on sphere with conditions}]
This follows immediately from Theorem \ref{thm: improved entropic on sphere general I} and the fact that
\begin{equation}\nonumber
N \leq \frac{H_N\pa{F_N}}{C_H}.
\end{equation}
\end{proof}
\begin{proof}[Proof of Theorem \ref{thm: imrpoved entropic on sphere with spherical fisher information}]
This follows immediately from Theorem \ref{thm: improved entropic on sphere general II}, the known inequality
\begin{equation}\label{eq: fisher inequality}
\sum_{j=1}^N I_N\pa{F^{\pa{N}}_j} \leq 2I_N\pa{F_N}
\end{equation}
(see \cite{BCD}) and, much like the proof of Theorem \ref{thm: improved entropic on sphere with conditions}, the fact that $N \leq \frac{H_N\pa{F_N}}{C_H}$.
\end{proof}
\section{A Non Trivial Example.}\label{sec: example}
As was mention in the introduction of this work, there is a connection between inequalities (\ref{eq: entropic inequality on the sphere}) and the subject of entropic convergence to equilibrium in Kac's model (the many body Cercignani's conjecture). It is thus not surprising that in order to find a family of density functions that will serve as an example to the validity of the conditions of our main theorems, we look for natural 'states' occurring in the setting of Kac's model. Such states, intimately connected to the concept of \emph{chaoticity} and \emph{entropic chaoticity} are described below (for more information we refer the reader to \cite{CCRLV,Carr,Einav1.5,HM,MM}).\\
Given $f\in P\pa{\R}$, with additional conditions we will mention shortly, we can define the \emph{normalisation function}, $\mathcal{Z}_N\pa{f,r}$, as
\begin{equation}\nonumber
\mathcal{Z}_N\pa{f,r}=\int_{\spa{S}{N-1}\pa{r}}f^{\otimes N}d\sigma^N_r.
\end{equation}
The \emph{conditioned tensorisation of $f$} on the sphere is the probability measure on $\kac$ with density
\begin{equation}\nonumber
F_N=\frac{f^{\otimes N}}{\mathcal{Z}_N\pa{f,\sqrt{N}}}
\end{equation}
The following theorem, proved in \cite{CCRLV}, is of great inportance in the study of conditioned tensorisations, and reinforces the intuition that when $f$ has a unit second moment the $N-$tensorisation function of $f$ is concentrated tightly about $\kac$.
\begin{theorem}\label{thm: CCRLV theorem}
Let $f\in P\pa{\R}$ such that $f\in L^p(\mathbb{R})$ for some $p>1$, $\int_\mathbb{R} v^2 f(v)dv=1$ and $\int_\mathbb{R} v^4 f(v)dv<\infty$. Then
\begin{equation}\label{eq: normalisation function fixed f approximation}
\mathcal{Z}_N(f,\sqrt{u})=\frac{2}{\sqrt{N}\Sigma \left\lvert \mathbb{S}^{N-1}\right\rvert  u^{\frac{N-2}{2}}}\left( \frac{e^{-\frac{(u-N)^2}{2N\Sigma^2}}}{\sqrt{2\pi}}+\lambda_N(u) \right),
\end{equation}
where $\Sigma^2 = \int_{\mathbb{R}}v^4f(v)dv - 1$ and $\sup_u \abs{\lambda_N(u)}\underset{N\rightarrow\infty}{\longrightarrow}0$.
\end{theorem}
We are now ready to present our non-trivial example for a family of densities on the sphere that satisfies the conditions of our main theorems. While extensions of it can be found, we restrict ourselves to a relatively simple case to avoid some lengthy computations.
\begin{theorem}\label{thm: eample}
Let $f\in P\pa{\R}\cap C_c\pa{\R}$, $f\not=\gamma$, be such that $\int_\mathbb{R} v^2 f(v)=1$ and $I(f)<\infty$. Then, the conditioned tensorisation of $f$ satisfies the conditions of Theorem \ref{thm: improved entropic on sphere with conditions} and \ref{thm: imrpoved entropic on sphere with spherical fisher information}.
\end{theorem}
\begin{proof}
The first thing we note is that since $F_N$ is symmetric with respect to its variables all the marginals are identical. As such, for any $j\geq 2$
\begin{equation}\nonumber
H_N\pa{\Pi_1^{\pa{j}}\pa{F_N}}=H_N\pa{\Pi_1^{\pa{1}}\pa{F_N}}=H_N\pa{\Pi_1\pa{F_N}}
\end{equation}
and the same holds for $I,I_N$ and $M_k$. In that case the appropriate averaged quantities, $\mathcal{A}$, are
\begin{equation}\nonumber
\mathcal{A}_k=\sup_N M_k\pa{\Pi_1\pa{F_N}},\; \mathcal{A}_I=\sup_N I\pa{\Pi_1\pa{F_N}},\; \mathcal{A}_q^P=\sup_N P_q^{\pa{1}}\pa{\Pi_1\pa{F_N}}.
\end{equation}
Using formula (\ref{eq: normal marginal formula}) and the definition of the normalisation function we have that
\begin{equation}\label{eq: conditioned tensorisation marginal}
\begin{gathered}
\Pi_1\pa{F_N}(v)=\frac{\abs{\mathbb{S}^{N-2}}\pa{1-\frac{v^2}{N}}^{\frac{N-3}{2}}\mathcal{Z}_{N-1}\pa{f,\sqrt{N-v^2}}}{\abs{\mathbb{S}^{N-1}}\sqrt{N}\mathcal{Z}_N\pa{f,\sqrt{N}}}f(v)\\
=\sqrt{\frac{N}{N-1}}\frac{e^{-\frac{\pa{1-v^2}^2}{2(N-1)\Sigma^2}}+\sqrt{2\pi}\lambda_{N-1}\pa{N-v^2}}{1+\sqrt{2\pi}\lambda_N(N)}f(v),
\end{gathered}
\end{equation}
due to (\ref{eq: normalisation function fixed f approximation}). As such
\begin{equation}\nonumber
\mathcal{A}_k=\sup_N M_k\pa{\Pi_1\pa{F_N}} \leq \sup_N \frac{1+\sqrt{2\pi}\sup \abs{\lambda_{N-1}}}{1+\sqrt{2\pi}\lambda_N(N)}\sqrt{\frac{N}{N-1}}\int_{\R}\abs{v}^k f(v)dv < \infty,
\end{equation}
for any $k>0$ as $f\in C_c\pa{\R}$. \\
Let $R>0$ be such that $f$ is supported in $[-R,R]$. We find that for $N>R$
\begin{equation}\nonumber
\begin{gathered}
\mathcal{A}_q^P=\sup_N P_q^{\pa{1}}\pa{\Pi_1\pa{F_N}} =\sup_N \int_{-R}^{R} \frac{\Pi_1\pa{F_N}(v)}{\pa{1-\frac{v^2}{N}}^{\frac{q}{q-2}}}dv \\
\leq \sup_N \frac{1+\sqrt{2\pi}\sup \abs{\lambda_{N-1}}}{1+\sqrt{2\pi}\lambda_N(N)} \sqrt{\frac{N}{N-1}}\frac{1}{\pa{1-\frac{R^2}{N}}^{\frac{q}{q-2}}} <\infty,
\end{gathered}
\end{equation}
for any $q>2$.\\
Using (\ref{eq: connection between fisher information on the sphere}) and the fact that $f$ is compactly supported, we see that for $N>R$ 
\begin{equation}\nonumber
\begin{gathered}
I\pa{\Pi_1\pa{F_N}} =\int_{-R}^R \abs{\frac{d}{dv}\log \Pi_1\pa{F_N}(v)}\Pi_1\pa{F_N}(v)dv  \\
\leq \frac{I_N\pa{F_1^{\pa{N}}}+2}{\pa{1-\frac{R^2}{N}}} \leq \frac{2\pa{\frac{I_N\pa{F_N}}{N}+1}}{\pa{1-\frac{R^2}{N}}},
\end{gathered}
\end{equation}
where we have used (\ref{eq: fisher inequality}) and the symmetry of $F_N$. This implies that 
\begin{equation}\nonumber
\mathcal{A}_I = \sup_N I\pa{\Pi_1\pa{F_N}} \leq \sup_N \frac{2\pa{\frac{I_N\pa{F_N}}{N}+1}}{\pa{1-\frac{R^2}{N}}},
\end{equation}
showing that if
\begin{equation}\nonumber
\sup_N \frac{I_N\pa{F_N}}{N}<\infty
\end{equation}
we obtain the required Fisher Information condition for Theorem \ref{thm: improved entropic on sphere with conditions}, as well as Theorem \ref{thm: imrpoved entropic on sphere with spherical fisher information}. We find that
\begin{equation}\nonumber
\begin{gathered}
\frac{I_N\pa{F_N}}{N}=\frac{1}{N}\int_{\kac} \frac{\abs{\nabla_{\mathbb{S}}F_N}^2}{F_N}d\sigma^N \leq \frac{1}{N}\int_{\kac} \frac{\abs{\nabla F_N}^2}{F_N}d\sigma^N\\
=\int_{\kac} \pa{\frac{f^\prime (v_1)}{f(v_1)}}^2 F_N d\sigma^N=\int_{\R} \pa{\frac{f^\prime (v)}{f(v)}}^2\Pi_1\pa{F_N}(v)dv
\end{gathered}
\end{equation}
\begin{equation}\nonumber
 \leq \sup_N \frac{1+\sqrt{2\pi}\sup \abs{\lambda_{N-1}}}{1+\sqrt{2\pi}\lambda_N(N)}\sqrt{\frac{N}{N-1}} I(f)=C_I<\infty,
\end{equation}
where we have used the special structure of $F_N$ and symmetry.\\
Last, but not least, we will deal with the rescaled entropy term.
\begin{equation}\nonumber
\begin{gathered}
\frac{H_N\pa{F_N}}{N}=\frac{1}{N}\int_{\kac} F_N \log f^{\otimes N} d\sigma^N - \frac{\log \mathcal{Z}_N\pa{f,\sqrt{N}}}{N}\\
=\int_{\R} \log f(v)\Pi_1\pa{F_N}(v)dv +\frac{\log\pa{\abs{\spa{S}{N-1}}N^{\frac{N-1}{2}}}}{N}-\frac{\log\pa{\frac{2}{\sqrt{2\pi}\Sigma}\pa{1+\sqrt{2\pi}\lambda_N(N)}}}{N}.
\end{gathered}
\end{equation}
As $f$ is supported on $[-R,R]$ we find that $\Pi_1\pa{F_N}$ converges to $f$ uniformly on $\R$. Also, using the asymptotic approximation of $\abs{\spa{S}{N-1}}$ one can show that
\begin{equation}\nonumber
\frac{\log\pa{\abs{\spa{S}{N-1}}N^{\frac{N-1}{2}}}}{N} \underset{N\rightarrow\infty}{\longrightarrow} \frac{1+\log\pa{2\pi}}{2}=-\int_{\R}f(v)\log \gamma(v) dv.
\end{equation}
Thus, 
\begin{equation}\nonumber
\lim_{N\rightarrow\infty}\frac{H_N\pa{F_N}}{N}=H\pa{f|\gamma}>0,
\end{equation}
and since $F_N\not\equiv 1$ we know that $H_N\pa{F_N}\not=0$ for all $N$, implying that there exists $C_H>0$ with 
\begin{equation}\nonumber
\frac{H_N\pa{F_N}}{N} \geq C_H,
\end{equation}
completing our theorem.
\end{proof}
\begin{remark}
Note that in the proof of the above theorem the only quantity that wasn't bounded by an 'explicit' constant is the rescaled entropy. However, such a constant can be found by a more detailed computation.
\end{remark}
\section{Final Remarks}\label{sec: final}
While the main result proven in this paper gives a glimpse of tools and quantities that are of import both to the equivalence of ensembles and many body Cercignani's conjecture, there are still many items of interest that can be explored in future research. We present a few remarks and observations related to that:
\begin{itemize}
\item The condition on the pole control, $P_q^{\pa{j}}$, seems to fit the problematic behaviour near the poles that was used to show that the constant in (\ref{eq: entropic inequality on the sphere}) is sharp. However, in relation to Kac's model, it seems hard to show the propagation of such property under Kac's flow. If one is allowed to use the exponent $q=\infty$, it is easy to see that the expression given for $P_\infty^{\pa{j}}$ is controlled by $I_N\pa{F_j^{\pa{N}}}$ - a more natural quantity in the kinetic setting. It would be interesting to see what will need to replace, if possible, the condition about infinite moment control (i.e. $k=\infty$) in order to be able to use this.
\item The moment control condition appears to be natural in Kac's setting. Indeed, following \cite{Einav1} one sees that the family of functions that was constructed to show the validity of Villani's conjecture satisfies
\begin{equation}\nonumber
M_k\pa{\Pi_1\pa{F_N}}\underset{N\rightarrow\infty}{\longrightarrow}\infty,
\end{equation}
for any $k>2$.
\item A very important observation, that can be made following Theorems \ref{thm: improved entropic on sphere general I} and \ref{thm: improved entropic on sphere general II}, is that the requirement on $\frac{H_N\pa{F_N}}{N}$ can be removed and one can gain a quantitative version of the deviation of the sum of the partial entropies with respect to the total entropy. In other words, we can find an explicit $\kappa_N$ such that
\begin{equation}\nonumber
\sum_{j=1}^N \int_{\kac}F_j^{\pa{N}}\log F_j^{\pa{N}}d\sigma^N \leq H_N\pa{F_N}+\kappa_N.
\end{equation}
Under our setting $\kappa_N$ may blow up but perhaps a more delicate estimation can be done in the future to evaluate it, or some regimes on the behaviour of $H_N\pa{F_N}$ may be explored and will allow us to improve our main inequality.
\item The rescaled entropy, $\frac{H_N\pa{F_N}}{N}$ is very important in the study of Kac's model and is connected to the concept of entropic chaoticity (see more in \cite{CCRLV, Einav2,HM,MM}). One knows that under Kac's flow the entropy will decrease, so a lower bound on the rescaled entropy can't propagate with time. However, it may give rise to a two time scale approach where we find a fast convergence to a state near equilibrium if we start far from equilibrium using the ideas in our work, followed by a fast convergence to equilibrium using different techniques.
\end{itemize}
\appendix
\section{Additional Proofs. }
In this Appendix we will provide additional proofs that we felt would hinder the flow of the paper.
\begin{lemma}\label{lem: entropic inequality on R}
Let $F_N\in P\pa{\R^N}$ a probability density with finite second moment. Then 
\begin{equation}\label{eq-app: entropi inequality on R}
\sum_{j=1}^N H\pa{\Pi_1^{\pa{j}}\pa{F_N}|\gamma} \leq H\pa{F_N | \gamma_N}.
\end{equation}
\end{lemma}
\begin{proof}
Note that 
\begin{equation}\nonumber
\begin{gathered}
\sum_{j=1}^N \int_{\R}\Pi_1^{\pa{j}}\pa{F_N}(v_j)\log \gamma(v_j) dv_j = -\frac{N\log 2\pi}{2}-\frac{1}{2}\sum_{j=1}^N \int_{\R^N}v_j^2 F_N(v)dv \\
=-\frac{N\log 2\pi}{2}-\frac{1}{2} \int_{\R^N}|v|^2 F_N(v)dv = \int_{\R^N} F_N(v)\log \gamma_N(v)dv.
\end{gathered}
\end{equation}
Thus, we only need to prove that
\begin{equation}\nonumber
\sum_{j=1}^N H\pa{\Pi_1^{\pa{j}}\pa{F_N}} \leq H\pa{F_N }.
\end{equation}
Define $G_N(v)=\Pi_{j=1}^N \Pi_1^{\pa{j}}\pa{F_N}(v_j)$. $G_N\in P\pa{\R^N}$ and 
\begin{equation}\nonumber
\begin{gathered}
0 \leq H\pa{F_N | G_N} = H\pa{F_N}-\int_{\R^N} F_N(v)\log G_N(v)dv \\
=H\pa{F_N}-\sum_{j=1}^N \int_{\R^N}F_N(v)\log\pa{\Pi_1^{\pa{j}}\pa{F_N}(v_j)}dv=H\pa{F_N}-\sum_{j=1}^N H\pa{\Pi_1^{\pa{j}}\pa{F_N}},
\end{gathered}
\end{equation}
completing the proof.
\end{proof}

\begin{lemma}\label{app: inequality for W_q connection to W_1}
Denote by $d(x,y)=\min\pa{\abs{x-y},1}$ for any $x,y\in \R$. Then for any $q\geq 1$ and $R\geq 1$
\begin{equation}\label{eq: inequality for W_q connection to W_1}
\abs{x-y}^q \leq R^q d(x,y) + \frac{2^k}{R^{k-q}}\pa{\abs{x}^k+\abs{y}^k}.
\end{equation}
\end{lemma}
\begin{proof}
If $\abs{x-y}\leq 1$ we have that
\begin{equation}\nonumber
\abs{x-y}^q \leq \abs{x-y}=d(x,y) \leq  R^q d(x,y) + \frac{2^k}{R^{k-q}}\pa{\abs{x}^k+\abs{y}^k}.
\end{equation}
When $\abs{x-y}>1$ we have that if $\abs{x},\abs{y}<\frac{R}{2}$ 
\begin{equation}\nonumber
\abs{x-y}^q \leq 2^{q-1} \pa{\abs{x}^q+\abs{y}^q} \leq R^q =R^q d(x,y) \leq R^q d(x,y) + \frac{2^k}{R^{k-q}}\pa{\abs{x}^k+\abs{y}^k},
\end{equation}
due to the convexity of the map $f(t)=t^q$. If $\abs{x}<\frac{R}{2}$ and $\abs{y}>\frac{R}{2}$ (or vice versa)
\begin{equation}\nonumber
\begin{gathered}
\abs{x-y}^q \leq 2^{q-1} \pa{\abs{x}^q+\abs{y}^q} \leq \frac{R^q}{2}+2^{q-1}\pa{\frac{2}{R}}^{k-q}\abs{y}^k  \\
=\frac{R^q}{2}d(x,y)+\frac{2^{k-1}}{R^{k-q}}\abs{y}^k \leq R^q d(x,y) + \frac{2^k}{R^{k-q}}\pa{\abs{x}^k+\abs{y}^k}.
\end{gathered}
\end{equation}
Lastly, if $\abs{x},\abs{y}\geq \frac{R}{2}$ then
\begin{equation}\nonumber
\begin{gathered}
\abs{x-y}^q \leq 2^{q-1} \pa{\abs{x}^q+\abs{y}^q} \leq 2^{q-1}\pa{\frac{2}{R}}^{k-q}\pa{\abs{y}^k+\abs{x}^k}  \leq R^q d(x,y) + \frac{2^k}{R^{k-q}}\pa{\abs{x}^k+\abs{y}^k},
\end{gathered}
\end{equation}
completing the proof.
\end{proof}
\begin{lemma}\label{app: holder type inequality}
Let $\br{a_{j,i}}_{i=1,\dots,N,\;j=1,\dots,m}$ be non-negative numbers. Let $p_1,\dots,p_m$ be positive numbers such that $\sum_{j=1}^m \frac{1}{p_j}\leq 1$. Then
\begin{equation}\label{eq-app: holder type inequality}
\sum_{i=1}^N \pa{\Pi_{j=1}^m a_{j,i}^{\frac{1}{p_j}}} \leq \Pi_{j=1}^m \pa{\frac{\sum_{i=1}^N a_{j,i}}{N}}^{\frac{1}{p_j}}N.
\end{equation}
\end{lemma}
\begin{proof}
It is sufficient to prove that
\begin{equation}\nonumber
\sum_{i=1}^N \pa{\Pi_{j=1}^m a_{j,i}^{\frac{1}{p_j}}} \leq N^{1-\sum_{j=1}^m \frac{1}{p_j}} \Pi_{j=1}^m \pa{\sum_{i=1}^N a_{j,i}}^{\frac{1}{p_j}},
\end{equation}
which we will do by induction on $m$. We note that since $\sum_{j=1}^m \frac{1}{p_j}\leq 1$ we have that $p_j\geq 1 $ for all $j$. For $m=1$ we have, by H\"older inequality
\begin{equation}\nonumber
\sum_{i=1}^N a_{i}^{\frac{1}{p}} \leq \pa{\sum_{i=1}^N a_i }^{\frac{1}{p}}\pa{\sum_{i=1}^N 1}^{1-\frac{1}{p}}=N^{1-\frac{1}{p}}\pa{\sum_{i=1}^N a_i }^{\frac{1}{p}}.
\end{equation}
Assume (\ref{eq-app: holder type inequality}) is true for $m$. Then, denoting by $q_1=\frac{p_1}{p_1-1}$ we have that
\begin{equation}\nonumber
\begin{gathered}
\sum_{i=1}^N \pa{\Pi_{j=1}^{m+1} a_{j,i}^{\frac{1}{p_j}}} \leq \pa{\sum_{i=1}^N a_{1,i}}^{\frac{1}{p_1}}\pa{\sum_{i=1}^N \pa{\Pi_{j=2}^{m+1} a_{j,i}^{\frac{q_1}{p_j}}}}^{\frac{1}{q_1}}  \\
\leq \pa{\sum_{i=1}^N a_{1,i}}^{\frac{1}{p_1}}\pa{N^{1-\sum_{j=2}^{m+1} \frac{q_1}{p_j}} \Pi_{j=2}^{m+1} \pa{\sum_{i=1}^N a_{j,i}}^{\frac{q_1}{p_j}}}^{\frac{1}{q_1}} 
=N^{1-\sum_{j=1}^{m+1} \frac{1}{p_j}} \Pi_{j=1}^{m+1} \pa{\sum_{i=1}^N a_{j,i}}^{\frac{1}{p_j}},
\end{gathered}
\end{equation}
as 
\begin{equation}\nonumber
\sum_{j=2}^{m+1}\frac{q_1}{p_j}=q_1\pa{\sum_{j=1}^{m+1}\frac{1}{p_j}-\frac{1}{p_1}}\leq q_1\pa{1-\frac{1}{q_1}}=1.
\end{equation}
completing the proof.
\end{proof}


\begin{thebibliography}{99}
\bibitem{BCD}
Barthe. F, Cordero-Erausquin. D, and Maurey, B:
 \emph{Entropy of spherical marginals and related inequalities.} J. Math. Pures Appl. (9) 86 (2006), no. 2, 89--99.

\bibitem{Eric}
Carlen E. A.:
\emph{The Rate of Local Equilibration in Kinetic Theory},
Prospects in Mathematical Physics, pp 71-88, Contemp. Math., 437, Amer. Math. Soc.
\bibitem{CLL}
Carlen E. A., Lieb E. H., Loss M.
\emph{A Sharp Analog of Young's Inequality on $\S^N$ and Related Entropy Inequalities.} J. Geom. Anal. 14 (2004), issue 3, 487--520.

\bibitem{CCRLV}
    Carlen E. A., Carvalho M. C., Le Roux J., Loss M. and Villani C.:
    Entropy and Chaos in the Kac Model.
    \emph{Kinet. Relat. Models}, \textbf{3} (2010), no. 1,  85--122.

\bibitem{Carr}
    Carrapatoso K.:
    Quantitative and Qualitative Kac's Chaos on the Boltzmann Sphere. http://arxiv.org/abs/1205.1241.
    
\bibitem{C-E}
Cordero-Erausquin D.:
\emph{Some applications of mass transport to Gaussian type inequalities.} Arch. Rational Mech. Anal. 161 (2002), 257-269.

\bibitem{Einav1}
    Einav A.:
    On Villani's Conjecture Concerning Entropy Production for the Kac Master Equation.
    \emph{Kinet. Relat. Models}, \textbf{4} (2011), no. 2, 479--497.
    
\bibitem{Einav1.5}
    Einav A.:
    A Counter Example to Cercignani's Conjecture for the $d-$Dimensional Kac Model.
    \emph{J. Stat. Phys.}. \textbf{148} (2012), no. 6, 1076--1103. 
        
\bibitem{Einav2}
  Einav A.:
  \emph{A Few Ways to Destroy Entropic Chaoticity on Kac's Sphere}. Comm. Math. Sci. \textbf{12} (2014), No. 1, 41--60.
  
\bibitem{HM}
 Hauray M. and Mischler S.:
\emph{On Kac's Chaos and Related Problems.} HAL:http://hal.archives-ouvertes.fr/hal-00682782/.   
    
    
\bibitem{MM}
    Mischler S. and Mouhot C.:
   Kac's Program in Kinetic Theory.
   Invent. Math. \textbf{193} (2013), no. 1, 1--147
    
\bibitem{Vtransport1}
   Villani C.:
   \emph{Topics in Optimal Transportation},
   Graduate Studies in Mathematics, Vol. 58, American Mathematical Society, Providence, RI, 2003.

\bibitem{Vtransport2}
   Villani C.:
   \emph{Optimal Transport, Old and New},
   Grundlehren der Mathematischen Wissenschaften, Vol. 338, Springer London, 2009.
 
\bibitem{Villani}
    Villani C.:
    Cercignani's Conjecture is Sometimes True and Always Almost True.
    \emph{Comm. Math. Phys.}, \textbf{234} (2003), no. 3, 455--490. 
\end{thebibliography}
\end{document}